\documentclass[12pt]{amsart}
\usepackage{amsfonts, amsmath, amstext, amssymb, amsthm}
\usepackage{tikz}
\usepackage{tikz-cd}
\usetikzlibrary{decorations.pathreplacing}
\usetikzlibrary{calc,tikzmark}
\usepackage{latexsym}
\usepackage{url}
\usepackage[margin = 1in]{geometry}
\usepackage{amsmath}
\usepackage{amssymb}
\usepackage{amsfonts}
\usepackage{amsthm}
\usepackage{mathtools}
\usepackage{enumerate}

\newcommand{\qq}{\mathfrak{q}}
\newcommand{\mm}{\mathfrak{m}}
\newcommand{\NN}{\mathbb{N}}

\newcommand{\OO}{\mathcal{O}}
\newcommand{\RR}{\mathbb{R}}
\newcommand{\Aa}{\mathbb{A}}
\newcommand{\QQ}{\mathbb{Q}}
\newcommand{\ZZ}{\mathbb{Z}}
\newcommand{\FF}{\mathbb{F}}
\newcommand{\pp}{\mathfrak{p}}

\newcommand{\PP}{\mathbb{P}}

\newcommand{\Pp}{\mathfrak{P}}

\DeclareMathOperator{\red}{red}
\DeclareMathOperator{\Gal}{Gal}
\DeclareMathOperator{\Cl}{Cl}
\DeclareMathOperator{\ch}{char}

\DeclareMathOperator{\GL}{GL}

\DeclareMathOperator{\Nrd}{Nrd}
\DeclareMathOperator{\Trd}{Trd}

\newtheorem{thm}{Theorem}[section]
\newtheorem{prop}[thm]{Proposition}
\newtheorem{lem}[thm]{Lemma}
\newtheorem{cor}[thm]{Corollary}
\theoremstyle{definition}
\newtheorem{defn}[thm]{Definition}
\newtheorem{rmk}[thm]{Remark}
\newtheorem*{rmk*}{Remark}
\newtheorem*{ex*}{Example}

\title[Nonnorms are diophantine]{Diophantine definability of nonnorms of cyclic extensions of
  global fields}

\author{Travis Morrison}
\address{Institute for Quantum Computing \\ The University of Waterloo
  \\ 20 University Ave West \\ Waterloo, Ontario N2L 3G1, Canada}
  \email{travis.morrison@uwaterloo.ca}

\thanks{The author was partially supported by National
    Science Foundation grants DMS-1056703 and CNS-1617802.}  
\begin{document}

\begin{abstract}
  We show that for any square-free natural number $n$ and any global
  field $K$ with $(\ch(K), n)=1$ containing a primitive $n$th root of unity,
  the pairs $(x,y)\in K^{\times}\times K^{\times}$ such that $x$ is not a relative norm of
  $K(\sqrt[n]{y})/K$ form a diophantine set over $K$. We use the Hasse
  norm theorem, Kummer theory, and class field theory to prove this
  result. We also prove that for any $n\in \NN$ and any global field $K$ with
  $\ch(K)\neq n$, $K^{\times}\setminus K^{\times n}$ is diophantine over
  $K$. For a number field $K$, this is a result of Colliot-Th\'el\`ene and
  Van Geel, proved using results on the Brauer-Manin obstruction.
 Additionally, we prove a variation of our main theorem for
  global fields $K$ without the $n$th roots of unity, where we
  parametrize varieties arising from norm forms of cyclic extensions
  of $K$ without any rational points by a diophantine set.
\end{abstract}
\maketitle
\section{Introduction}

The diophantine subsets of a field $K$ which is not algebraically
closed are the subsets which are
defined by a positive-existential formula. 

\begin{defn}
  Let $R$ be a commutative domain. A set $A\subset R^n$ is diophantine
  over $R$ if there exists $m\in \NN$ and a polynomial
  $f\in R[x_1,\ldots,x_m,y_1,\ldots,y_n]$ such that
\[
  A =\{(a_1,\ldots,a_n): \exists r_1,\ldots,r_m\in R \text{ such that }
  f(r_1,\ldots,r_m,a_1,\ldots,a_n)=0\}.
 \]
 \end{defn}
 More geometrically, a
subset $A\subseteq K^n$ is diophantine over $K$ if there is an affine
algebraic set $X/K$ and a $K$-morphism $X\to \Aa^n_K$ such that $A$ is the image of $X(K)$. 

 Diophantine sets feature prominently in decidability and definability
 in number theory. Hilbert's Tenth Problem, abbreviated H10, asked
 whether there is an algorithm which takes as input an arbitrary
 polynomial equation with coefficients in $\ZZ$ and outputs YES if
 that equation has a solution over $\ZZ$ and NO
 otherwise. In~\cite{Mat70}, Matiyasevich, building on the work of
 Davis, Putnam and Robinson~\cite{DPR61}, proves such an algorithm
 cannot exist, i.e,
 H10 is undecidable. This is a consequence of the ``DPRM'' theorem: the
 diophantine subsets of $\ZZ$ are precisely the recursively enumerable
 subsets of $\ZZ$. We can then ask if such an algorithm exists for
 other rings $R$ by replacing $\ZZ$ with $R$. Over $\QQ$, H10 is still
 open. If one could show that $\ZZ$ were diophantine over $\QQ$, then
 a standard argument reduces H10 over $\ZZ$ to H10 over $\QQ$,
 implying that H10 over $\QQ$ would be undecidable as
 well. The DPRM theorem resolved Hilbert's tenth problem by
 classifying the diophantine subsets of $\ZZ$, which suggests that
 understanding the sets which are diophantine over a global field $K$
 sheds some light on the difficulty of solving diophantine equations
 over $K$. 

Obstructions to the existence of rational points on a variety can be used to
produce diophantine
definitions of sets. Poonen, in \cite{Poo09b}, uses
results on the Brauer-Manin obstruction to show that the non-squares
of a global field $K$  of characteristic not $2$ are diophantine over
$K$. With similar methods, in \cite{VAV2012}, 
V\'arilly-Alvarado and Viray show that, assuming Schinzel's hypothesis, for any natural
number $n$ and number field $K$, the set of non-$n$th-powers of $K$
is diophantine over $K$.  
Colliot-Th\'{e}l\`{e}ne and Van Geel unconditionally prove this result
in \cite{CTG15}. This is further generalized in~\cite{Dit16}, where
Dittman shows that the irreducibility of polynomials over a global
field is diophantine. 

Because H10 is undecidable over $\ZZ$, a first-order definition of $\ZZ$ in $\QQ$ shows that the full first-order theory of $\QQ$ is undecidable. Robinson gave a first-order definition of $\ZZ$ in $\QQ$, showing that it has a $\forall\exists\forall\exists$ definition in~\cite{Rob49}, and Poonen improved on this in~\cite{Poo09}, showing that $\ZZ$ has a $\forall\exists$ definition in $\QQ$. Koenigsmann proves that $\QQ\setminus \ZZ$ is diophantine over $\QQ$ and moreover that this implies that $\ZZ$ has a $\forall$ definition in $\QQ$. 
Using similar methods,  Koenigsmann gives a new proof that $\QQ^{\times}\setminus
\QQ^{\times 2}$ is diophantine over $\QQ$, and also shows that the set  
\[
\{(x,y): x\text{ is not a norm of } \QQ(\sqrt{y})\}
\]
is diophantine over $\QQ$. Park shows that for a number field $K$,
$\OO_K$ has a first-order universal definition. In \cite{EM17}, these
results are generalized further: for a global field $K$ with
$\ch(K)\not=2$ and $S$ a finite, nonempty set of primes of $K$, the sets
$K\setminus \OO_S$, $K\setminus K^{\times 2}$, and
\[
\{(x,y): x\text{ is not a norm of } K(\sqrt{y})\}
\]
 are shown to
be diophantine over $K$.

In this paper, we prove the
following theorem:
\begin{thm}\label{thm:nonnormprime}
  Let $\ell$ be an odd prime and suppose that $K$ is a global field
  with $\ch(K)\not=\ell$. Further suppose that $K$ contains a
  primitive $\ell$th root of unity. Then
\[
\{(x,y)\in K^{\times}\times K^{\times}: x \text{ is not a norm of } K(\sqrt[\ell]{y}) \}
\]
is diophantine over $K$.
\end{thm} 
We prove Theorem \ref{thm:nonnormprime} in Section
\ref{sec:mainthm}. We use the Hasse norm theorem, which says that
$x\in K$ is a relative norm of a cyclic extension $L/K$ if and only if
it is a relative local norm in every completion of $L/K$.  For a fixed
prime of $K$, whether $x$ is a relative norm of
$K_{\pp}(\sqrt[\ell]{y})$ is 
controlled by diophantine local conditions on $x$ and $y$. This is not enough to
prove Theorem~\ref{thm:nonnormprime}: while finite unions of
diophantine sets are diophantine, infinite unions do not have to be
diophantine. To reduce from infinitely many to finitely many
conditions on $x$ and $y$, we group the primes of $K$ into finitely
many ray classes for an abelian extension $L/K$, and use Kummer theory
and class field theory to relate the splitting of primes in these
classes to the Hilbert symbol.

Next, we show that for each union of ray classes $C$, the set
\[
\{(x,y)\in K^{\times}\times K^{\times}: \exists \pp \in C \text{ such
  that } x \text{ is not a norm of } K_{\pp}(\sqrt[\ell]{y})\}
\]
is diophantine over $K$. Taking the union over these subsets of
$K^{\times}\times K^{\times}$ gives our diophantine definition of
$\{(x,y): x \text{ is not a norm of } K(\sqrt[\ell]{y})\}$.  Our approach to this is as follows.  We encode the local conditions
using certain semi-local subrings of $K$, defined by the norm and
trace forms of cyclic algebras. These semi-local rings are diophantine
by a result of~\cite{Dit16}. This is an extension of ideas in
\cite{Eis05,Poo09,Koe13,Park,EM17}, where quaternion algebras are used
to produce diophantine definitions of semi-local rings.  We define a
finite number of families of diophantine semi-local rings, where each
family depends on one or two parameters of $K$. Then we parametrize
these families by sets which are diophantine over $K$ and which ensure
that the primes of any semi-local ring in a family are all in the same
union of ray classes. 

As a
corollary to Theorem~\ref{thm:nonnormprime}, we obtain
\begin{cor}\label{cor:nonnormsqfree}
  Suppose $n>1$ is a square-free integer, $K$ is a global field
  containing a primitive $n$th root of
  unity, and assume
  $(\ch(K),n)=1$ if $\ch(K)>0$.  Then
\[
\{(x,y): x \text{ is not a norm of } K(\sqrt[n]{y}) \}
\]
is diophantine over $K$.
\end{cor}

As an additional corollary to Theorem~\ref{thm:nonnormprime}, we prove
\begin{cor}\label{cor:nonnthpowersdio}
Suppose $n>1$ is a square-free integer. Let $K$ be a global field such that
  $(\ch(K),n)=1$ if $\ch(K)>0$. Then
$K^{\times }\setminus K^{\times  n}$ is diophantine over $K$. 
\end{cor}

Thus we recover the result of~\cite{CTG15} in the case that $K$ is a
number field, and extend it to the case that $K$ is a global function
field with $(\ch(K),n)=1$.  We replace the use of the Brauer-Manin
obstruction with class field theory to
prove Corollary~\ref{cor:nonnthpowersdio}.  


Finally, we reinterpret Corollary~\ref{cor:nonnormsqfree} in order to
remove the assumption that $K$ contains a primitive root of unity in
our results on non-norms. Suppose $K$ is a global field, $n>1$ is a
square-free number, and $(\ch(K),n)=1$ if $\ch(K)>0$.  Let $d:=\binom{2n-1}{n}$ be
the dimension of the space of homogeneous polynomials of degree $n$ in
$n$ variables over $K$. For a cyclic extension $L/K$ of degree $n$,
the norm form is such a polynomial. Given a vector
$\vec{a}=(a_1,\ldots,a_d)\in K^d$, denote by $f_{\vec{a}}$ the
homogeneous polynomial of degree $n$ in the $n$ variables
$t_1,\ldots,t_n$ whose $i$th coefficient, using the lexicographical
ordering, is $a_i$. In Section~\ref{sec:norootsof1}, we prove the
following:

\begin{thm}\label{thm:norootsof1}
Suppose $n>1$ is square-free and $K$ is a global field with $(\ch(K),n)=1$ if $\ch(K)>0$. The set 
\begin{align*}
D(n,K) :=\{(x,\vec{a})\in K^{\times }\times K^d \colon 
&f_{\vec{a}} \text{ is the norm form of a degree } n \text{ cyclic extension of } K \\
&\text{ and } f_{\vec{a}}(t_1,\ldots,t_n) = x \text{ has no solutions in } K^n \}
\end{align*}
is diophantine over $K$. 
\end{thm}

\section{Class Field Theory, the Hilbert Symbol, 
and Cyclic Algebras}
For a number field $K$, a {\em finite prime} $\pp$ of $K$ is a maximal
ideal in $\OO_K$, the ring of integers of $K$, and an {\em infinite prime} of
$K$ is an equivalence class of Archimedean absolute values. If $K$ is
a global function field, a {\em finite prime} $\pp$ is the maximal
ideal of a local ring in $K$. If $\pp$ is a finite prime of a global
field $K$, let $v_{\pp}:K\to \ZZ\cup\{\infty\}$ be the associated
normalized valuation. For a global field $K$ and a finite prime $\pp$,
a local ring associated to $\pp$ is denoted by
$\OO_{\pp}:=\{x\in K: v_{\pp}(x)\geq 0\}$. A {\em semi-local ring} in
$K$ is a finite intersection of local rings of $K$. Local rings of
global fields are diophantine by the following lemma, first proved
in~\cite{Shl94}:
\begin{lem}\label{Opdio}
  Let $K$ be a global field and let $\OO_{\pp}\subseteq K$ be a
  local ring in $K$. Then $\OO_{\pp}$ is diophantine over $K$.
\end{lem}
\begin{proof}
See \cite[Lemma 3.22]{Shl94}. 
\end{proof}

If $K$ is a global function field and $S$ is
a finite, nonempty set of primes of $K$, the {\em ring of $S$-integers} is 
\[
\OO_S:=\bigcap_{\pp\not\in S} \OO_{\pp}
\]
Then $\OO_S$ is a dedekind domain, and the primes of $K$ not contained
in $S$ are in one-to-one correspondence with the maximal ideals in
$\OO_S$ by the map $\pp\mapsto \pp\cap\OO_S$. Set $A:=\OO_S$; then
there is a positive-characteristic analogue of the Hilbert class field
of $K$, denoted $K^A$. The extension $K^A$ is defined to be the
maximal abelian extension of $K$ in which every prime of $S$ splits
completely. The extension $K^A/K$ has finite degree over $K$ and
satisfies $\Cl(A)\simeq \Gal(K^A/K)$; see \cite{Rosen87}. For
additional background on arithmetic in global function fields, see
\cite{Rosen02}.
\subsection{The Artin Map}


A {\em modulus} $\mm = \prod_{\pp} \pp^{\mm(\pp)}$ of $K$ is a formal product of finitely many nonnegative powers of
primes (both finite and infinite) of $K$. Given a modulus $\mm$ of
$K$, let $\mm_0$ denote the finite part of $\mm$, and let $I(\mm)$
denote the free abelian group generated by the finite primes of $K$
such that  $\mm(\pp)=0$, i.e. those 
which do not divide $\mm$. The {\em support} of an element $\prod \pp^{e_{\pp}}\in I(\mm)$ are the primes $\pp$ such that $e_{\pp}\not=0$. Two elements of $I(\mm)$ are {\em coprime} if their supports are disjoint. Define 
\[
K_{\mm}:=\{x\in K: v_{\pp}(x)=0 \text{ for all }  \pp|\mm_0\}.
\]
 If $\pp|\mm$ is a real infinite prime, then it is associated with an embedding $K\to \RR$. An element of $K^{\times}$ is {\em  positive at $\pp$} if its image under the embedding associated to $\pp$ is positive. We define
\[
  K_{\mm,1}:=\{x\in K^{\times }: v_{\pp}(x-1)\geq \mm(\pp) \,\forall \pp|\mm_0
  \text{ and positive at each real } \pp|\mm\}.
\]
We have
an embedding  
\begin{align*}
 K_{\mm} &\to I(\mm) \\ 
x&\mapsto (x):=\prod_{\pp} \pp^{v_{\pp}(x)}.
\end{align*}
 Denote the image of $K_{\mm,1}$ in
$I(\mm)$ by $P(\mm)$ and set $C(\mm):=I(\mm)/P(\mm)$. 

Now suppose $L/K$ is an abelian extension and that $\mm$ is a modulus
of $K$ containing the primes of $K$ ramified in $L$. Given $\pp$
coprime to $\mm$, we denote the Frobenius of $\pp$ in $\Gal(L/K)$ by
$(\pp,L/K)$ and we define the Artin map for $L/K$:
\begin{align*}
\psi_{L/K}: I(\mm)&\to \Gal(L/K) \\
\prod \pp_i^{e_i} &\mapsto \prod (\pp,L/K)^{e_i}.
\end{align*}

Let $I_K:=I(1)$ denote the fractional ideals of $K$. Suppose
that $L/K$ is a finite extension. Given a prime $\Pp$ of $L$, set
$\pp:=\Pp\cap K$ and $f(\Pp|\pp):=[\OO_{\Pp}/\Pp:\OO_{\pp}/\pp]$. Then
we have the relative ideal norm map
\[
N_{L/K}(\Pp):=\pp^{f(\Pp|\pp)}
\]
which we extend to $I_L$. Artin Reciprocity lets us characterize the
kernel of the Artin map:
\begin{thm}[Artin Reciprocity]
  Suppose $L/K$ is a finite abelian extension. Then there is a modulus
  $\mm$ of $K$ such that if $\mm'$ is the modulus of $L$ containing
  the primes of $L$ above those dividing $\mm$, 
\[
\psi_{L/K}: I(\mm)/(P(\mm)\cdot N_{L/K}(I(\mm'))) \simeq \Gal(L/K)
\]
is an isomorphism. 
\end{thm}
Such a modulus $\mm$ is called an {\em admissible
modulus of $K$ for $L$}.
\subsection{Ray class groups of rings of $S$-integers}
Suppose that $K$ is a global function field and $S$ is a finite set of
primes of $K$. Set $A:=\OO_S$. Suppose $\mm$ is a modulus of $K$
divisible by primes which are not in $S$. Then
$M:=\prod_{\pp|\mm}(\pp\cap A)^{v_{\pp}(\mm)}$ is
an ideal of $A$. Let $I_{\mm}(A)$ be the abelian group generated by
primes of $A$ which are coprime to $M$, let $P_{\mm}(A)$
be the image of $K_{\mm,1}$ in $P_{\mm}(A)$, and let
$\Cl_{\mm}(A):=I_{\mm}(A)/P_{\mm}(A).$ We have the following ``folk
theorems;'' for proofs, see Corollary 2.8 and Theorem 2.9 in \cite{EM17}.
\begin{thm}\label{folkthms}
There is a finite abelian
extension $K_{\mm}^A/K$ such that $\Cl_{\mm}(A)\simeq \Gal(K_{\mm}^A/K)
$, where the isomorphism is the Artin map. Additionally, every class in $\Cl_{\mm}(A)$
contains infinitely many primes of $A$.
\end{thm}
We will need Theorem~\ref{folkthms} in Section~\ref{sec:integrality}.

\subsection{The power residue symbol}
Now assume $K$ is a global field, $\ell$ is a prime number coprime to the characteristic of $K$, and that $K$
contains $\mu_{\ell}$, the $\ell$th roots of unity. Let
$\omega\in\mu_{\ell}$ be a primitive root of unity. Let $a\in K^{\times }$. If $\pp$ is
a finite prime of $K$ such that $v_{\pp}(a)=0$, and if
$\alpha\in K(\sqrt[\ell]{a})$ is any root of $x^{\ell}-a$, then
$(\pp,K(\sqrt[\ell]{a})/K)(\alpha)/\alpha$ is an $\ell$th root of
unity and thus equals $\omega^s$ for some $0\leq s \leq \ell-1$. The
value of $s$ is independent of the choice of $\alpha$. Define
\[
\left(\frac{a}{\pp}\right)_{\ell} :=
\frac{(\pp,K(\sqrt[\ell]{a})/K)(\alpha)}{\alpha}
\]
to be the {\em $\ell$th power residue symbol} for the prime
$\pp$. Equivalently, $\omega^s=\left(\frac{a}{\pp}\right)_{\ell}$ is
the unique $\ell$th root of unity satisfying 
\[
\alpha^{(|\FF_{\pp}|-1)/\ell} \equiv \omega^s \pmod \pp.
\]
This lets us compute the Frobenius of a prime $\pp$ coprime to $\mm$. 
The power residue symbol is multiplicative on $\OO_{\pp}^{\times}$, and if
$\left(\frac{a}{\pp}\right)_{\ell}=1$, by Hensel's Lemma, there exists $b\in
\OO_{\pp}^{\times}$ such that $a\equiv b^{\ell} \pmod \pp$.  

We will use a fixed compositum of two cyclic degree $\ell$ extensions of $K$ in our diophantine definitions. Suppose $\ell$ is a prime number and $(\ch(K),\ell)=1$. Assume that
$a,b$ generate distinct, nontrivial subgroups in $K^{\times }/K^{\times  \ell}$
and suppose $\mm$ is a modulus of $K$ containing the primes of $K$
ramified in $L:=K(\sqrt[\ell]{a},\sqrt[\ell]{b})$. Then we have an isomorphism
\begin{align*}
  \iota:\Gal(L/K) &\to \mu_{\ell}\times\mu_{\ell} \\ 
  \sigma
            &\mapsto 
\left(\frac{\sigma(\sqrt[\ell]{a})}{\sqrt[\ell]{a}},
\frac{\sigma(\sqrt[\ell]{b})}{\sqrt[\ell]{b}}\right).
\end{align*}
 Thus for $L/K$, under the above identification, the Artin map is given by
\[
\iota((\pp,L/K)) =  
\left(\left(\frac{a}{\pp}\right)_{\ell},\left(\frac{b}{\pp}\right)_{\ell}\right).
\]



\subsection{The Hilbert Symbol}
Given a finite prime $\pp$ of $K$, we denote by $K_{\pp}$ the
completion of $K$ at $\pp$. Suppose $(\ch(K), n)=1$. By Kummer
theory and local class field theory, we have a non-degenerate pairing,
called the Hilbert symbol:
\[
(\cdot,\cdot)_{K_{\pp},n}:K_{\pp}^{\times }/K_{\pp}^{\times  n} \times K_{\pp}^{\times }/K_{\pp}^{\times  n} \to \mu_n 
\]
defined by 
\[
(\cdot,\cdot)_{K_{\pp},n} :=
\frac{(a,K_{\pp}(\sqrt[n]{b})/K_{\pp})(\sqrt[n]{b})}{\sqrt[n]{b}}.
\]

This pairing satisfies $(a,b)_{K_{\pp},n}=1$
if and only if $a$ is a norm in $K(\sqrt[n]{b})$. We also have the
following identity:
\begin{thm}[Hilbert Reciprocity]
For $a,b\in K$, 
\[
\prod_{\pp} (a,b)_{K_{\pp},n} = 1.
\]
\end{thm}
Let $a,b\in K_{\pp}$. Let $R_{\pp}\subseteq K_{\pp}$ be the ring of
integers of $K_{\pp}$, and let $\FF_{\pp}:=R_{\pp}/\pp$ be the residue
field of $\pp$. Let $\red_{\pp}: R_{\pp} \to \FF_{\pp}$ be the
reduction map.  We have the following formula for computing the
Hilbert symbol when $(\ch(\FF_{\pp}),n)=1$.
\begin{equation}\label{hilbertformula}
(a,b)_{K_{\pp},n} =
 \left((-1)^{v_{\pp}(a)v_{\pp}(b)} \red_{\pp}\left(\frac{a^{v_{\pp}(b)}}{b^{v_{\pp}(a)}}\right)\right)^{(|\FF_{\pp}|-1)/n}.
\end{equation}
\begin{proof}
This is Corollary to Proposition 8 in XIV.3 in \cite{Ser79}. 
\end{proof}
\subsection{Cyclic Algebras}

\begin{defn}
Let $n\in \NN$, let $K$ be a global field, and let $L/K$ be a finite cyclic extension of degree $n$. Let
$\sigma$ be a generator of $\Gal(L/K)$ and let $b\in K^{\times }$. The cyclic
 algebra associated with $\sigma$ and $b$, denoted $(\sigma,b)$, 
is generated by $L$ and an element $T\in (\sigma,b)$ which satisfies 
\[
T^n=b, \quad T\cdot s = \sigma(s)\cdot T
\]
for all $s\in L$. 
\end{defn}
This is a degree $n$ central simple algebra over $K$ containing $L$ as
a commutative sub algebra, and is split by $L$. If $(\ch(K),n)=1$ and
if $K$ contains the $n$th roots of unity $\mu_n$, we can give a
simpler presentation of a degree $n$ cyclic algebra. Given
$a,b\in K^{\times}$ and a primitive $n$th root of unity $\omega$, we define
\[
(a,b)_{\omega}:=\langle S,T \,|\, T^n=a,  S^n = b, ST=\omega TS\rangle.
\]
We are interested in these cyclic algebras because they are split if and only if their corresponding norm equations have rational solutions. We record the following theorem:
\begin{thm}
If $K$ is a global field  and $\sigma$ generates
$\Gal(L/K)$ for a cyclic extension $L/K$, the cyclic algebra $(\sigma,b)$ is
split if and only if $b$ is a norm of $L/K$. If $(\ch(K),n)=1$ and $K$ contains a
primitive $n$th root of unity $\omega$, then $(a,b)_{\omega}$ is split if and
only if $b$ is a norm of $K(\sqrt[n]{a})$.  
\end{thm}
\begin{proof}
See \cite{GS2006}, Corollaries 4.7.5 and 4.7.7. 
\end{proof}

We will also need the following theorem, proved in~\cite{Park}, on prescribing Hilbert
symbols:
\begin{thm}\label{prescribesymbols}
Let $n$ be a natural number and let $K$ be a global field containing
$\mu_n$ and satisfying $(\ch(K),n)=1$. Let $\Sigma$ denote the set
of primes of $K$, 
and let $\Lambda$ be a finite set of indices. Let $(a_i)_{i\in
  \Lambda}$ be a finite sequence of elements of $K^{\times}$ and suppose that
$(\varepsilon_{i,\pp})_{i\in \Lambda, \pp\in \Sigma}$ is a family of
elements of $\mu_{n}$. There exists $x\in K^{\times}$ satisfying
$(a_i,x)_{K_{\pp},n}=\varepsilon_{i,\pp}$ for all $i\in \Lambda$
and $\pp\in \Sigma$ if and only if the following conditions hold:
\begin{enumerate}
\item All but finitely many of the $\varepsilon_{i,\pp}$ are equal to
  1.
\item For all $i\in \Lambda$, we have $\prod_{\pp\in \Sigma}
  \varepsilon_{i,\pp}=1$. 
\item For every $\pp\in \Sigma$, there exists $x_{\pp}\in K^{\times}$ such
  that $(a_i,x_{\pp})_{K_{\pp},n}=\varepsilon_{i,\pp}$. 
\end{enumerate}
\end{thm}
\begin{proof}
Theorem 3.7 of~\cite{Park} is this theorem for $n=2$. For general $n$ and $K$ containing $\mu_n$ with characteristic coprime to $n$, the same proof carries through, using Proposition 3.5
and Lemma 3.6 of~\cite{Park}.  
\end{proof}

\section{Cyclic algebras and diophantine definitions of semi-local rings of $K$}\label{sec:diosemilocal}

Throughout this section, assume $\ell$ is a prime, $K$ is a global field,
$\ch(K)\not=\ell$, and that $\omega\in K$ is a primitive $\ell$th root
of unity. We then have, for $a,b\in K^{\times}$ and a prime $\pp$ of $K$,
\[
(a,b)_{K_{\pp},\ell}=1 \iff (a,b)_{\omega}\otimes K_{\pp} \text{ is split}. 
\]

Let $F$ be a splitting field for $A:=(a,b)_{\omega}$. For example, we can take
$F=K(\sqrt[\ell]{b})$. We then have $A\otimes F \simeq M_{\ell}(F)$. Given
$a\in A$, define the reduced norm $\Nrd(a)$ and reduced trace
$\Trd(a)$ to be the norm and trace respectively of the image of $a$ in
$M_{\ell}(F)$. If $a=\sum_{i,j=0}^{\ell} a_{ij}S^iT^j$ then $\Nrd(a)$ and
$\Trd(a)$ are polynomials in the coefficients $a_{ij}$ of $a$. Set 
\[
S_A:= \{\Trd(a): a\in A \text { and } \Nrd(a)=1\}. 
\]

\begin{defn}\label{def:Tab}
Let $a,b\in K^{\times}$ and define
\[
\Delta_{a,b}:=\{\pp: (a,b)_{\omega}\otimes K_{\pp}\not\simeq
M_{\ell}(K_{\pp}).\}
\]
Also define the semi-local ring
\[
T_{a,b}:=\bigcap_{\pp\in \Delta_{a,b}} \OO_{\pp}.
\]
\end{defn}

$\Delta_{a,b}$ is the set of primes where $(a,b)_{\omega}$ is not
split; the set $\Delta_{a,b}$ is finite. The following proposition,
which is a special case of Proposition 2.7 of~\cite{Dit16} but which
we record here for the reader, gives a diophantine definition of
$T_{a,b}$:

\begin{prop}\label{prop:Tabdio}
Assume that $\ell\in \NN$ is a prime, that $K$ is a global field with
$\ch(K)\neq \ell$, and $\omega\in K$ is a primitive $\ell$th root of
unity. There exists $B\in \NN$ such that if $R\subseteq K$ is a finite set of
representatives of 
\[
\bigcup_{\pp:|\FF_{\pp}|<B} \OO_{\pp}/\pp\OO_{\pp},
\]
then 
\[
T_{a,b}=S_{(a,b)_{\omega}}+S_{(a,b)_{\omega}}+R.
\]
Thus $T_{a,b}$ is diophantine over $K$. 
\end{prop}

\begin{proof}
This follows immediately from Proposition 2.7 of \cite{Dit16}.
\end{proof}

\subsection{Diophantine semi-local rings and their Jacobson radicals}
Again, fix a prime number $\ell$, a global field $K$ such that
$\ch(K)\neq \ell$, and a primitive $\ell$th root of unity
$\omega\in K$. In this section, we will show that certain semi-local
rings are diophantine. We will also show that their Jacobson radicals
contain an ideal which is diophantine. 

\begin{lem}\label{vpdivbyell}
Let $a,b\in K^{\times}$. Then 
\[
K^{\times \ell} \cdot T_{a,b}^{\times} = \bigcap_{\pp\in \Delta_{a,b}} v_{\pp}^{-1}(\ell\ZZ).
\]
\end{lem}
\begin{proof}
Suppose $x\in K^{\times \ell} \cdot T_{a,b}^{\times}$. Then we can write
$x=t^{\ell}\cdot u$ with $u\in T_{a,b}^{\times}$. Thus for any $\pp\in \Delta_{a,b}$, 
\[
v_{\pp}(x)=\ell \cdot v_{\pp}(t)+v_{\pp}(u)=\ell \cdot v_{\pp}(t).
\]
Now suppose that $x\in \bigcap_{\pp\in \Delta_{a,b}}
v_{\pp}^{-1}(\ell\ZZ)$. Then for all $\pp\in \Delta_{a,b}$, there exists
$k_{\pp}\in \ZZ$ such that $v_{\pp}(x) = \ell\cdot
k_{\pp}$. By weak approximation, 
there exists $t\in K^{\times}$ such that $v_{\pp}(t)=k_{\pp}$ for each $\pp\in
\Delta_{a,b}$. Then $v_{\pp}(x/t^{\ell})=0$ for each $\pp\in
\Delta_{a,b}$, so $u:=x/t^{\ell}\in T_{a,b}^{\times}$. Thus $x\in
K^{\times \ell} T_{a,b}^{\times}$. 
\end{proof}

Now suppose $a,b,c\in K^{\times}$ and define
\[
I_{a,b}^c:= c\cdot K^{\times\ell} \cdot T_{a,b}^{\times} \cap (1-K^{\times\ell} T_{a,b}^{\times}),
\]
and for $p\in K^{\times}$, define
\[
\PP(p):=\{\pp: v_{\pp}(p)\not\equiv 0 \mod \ell \}.
\]
\begin{lem}\label{IabcCalculation}
Suppose that $a,b,c\in K^{\times}$. Then
\begin{align*}
I_{a,b}^c= \{ x\in K &: v_{\pp}(x) \text{ is positive and } v_{\pp}(x)\equiv
                        v_{\pp}(c) \pmod{\ell} \text{ for }\pp \in \Delta_{a,b}\cap\PP(c), \\
&\text{ and } v_{\pp}(x),v_{\pp}(1-x)\equiv 0 \pmod{ \ell} \text{ for } \pp\in \Delta_{a,b}\setminus \PP(c)\}.
\end{align*}
\end{lem}
\begin{proof}
First we will show that $I_{a,b}^c$ is contained in the right-hand
side. Let $x\in I_{a,b}^c$. For all $\pp\in
\Delta_{a,b}$, both $v_{\pp}(x/c)$ and $v_{\pp}(1-x)$ are divisible by
$\ell$ by the previous lemma. For any $\pp\in \PP(c)$, this implies $v_{\pp}(x)\equiv v_{\pp}(c)\not\equiv 0 \pmod
\ell$. In particular, $v_{\pp}(x)\not=0$, so
$v_{\pp}(1-x)=\min\{0,v_{\pp}(x)\}$. Because 
$v_{\pp}(1-x)\equiv 0 \pmod \ell$, it follows that
$v_{\pp}(x)>0$. If $\pp\not \in \PP(c)$, then $v_{\pp}(x)\equiv
v_{\pp}(c)\equiv 0 \pmod \ell$. 

Conversely, suppose $x$ is in the right-hand side. First assume
$\pp\in \Delta_{a,b}\cap\PP(c)$. Then
$v_{\pp}(x)\equiv v_{\pp}(c) \pmod{\ell}$. We also have
$v_{\pp}(1-x)=0$, since $v_{\pp}(x)>0$.  For any prime
$\pp\in \Delta_{a,b}\setminus \PP(c)$, we have
$v_{\pp}(x/c)\equiv 0\pmod \ell$. It follows that
$x\in c\cdot K^{\times\ell}\cdot T_{a,b}^{\times}$ by
Lemma~\ref{vpdivbyell}. We also have that
$v_{\pp}(1-x)\equiv 0 \pmod \ell$ for all
$\pp\in \Delta_{a,b}\setminus \PP(c)$, so
$1-x\in K^{\times \ell}\cdot T_{a,b}^{\times}$ again by
Lemma~\ref{vpdivbyell}.
\end{proof}
\begin{defn}\label{Jabdefforell}
For $a,b\in K^{\times}$, define
\[
J_{a,b}:=\bigcap_{\pp\in \Delta_{a,b}\cap(\PP(a)\cup\PP(b))} \pp^{k_{\pp}} \OO_{\pp},
\]
where $1\leq k_{\pp} \leq \ell-1$ is defined by 
\[
k_{\pp} = \max\left\{ v_{\pp}(a) - \ell\left\lfloor\frac{v_{\pp}(a)}{\ell}\right\rfloor,
v_{\pp}(b)-\ell\left\lfloor\frac{v_{\pp}(b)}{\ell}\right\rfloor\right\}.
\]
\end{defn}
\begin{lem}\label{lem:Jablemma}
Let $a,b,c\in K^{\times}$. For $\pp\in \Delta_{a,b}\cap \PP(c)$, set $r_{\pp}: = v_{\pp}(c) - \ell\lfloor\frac{v_{\pp}(c)}{\ell}\rfloor.$ Then
\[
I_{a,b}^c+I_{a,b}^c=\bigcap_{\pp\in\Delta_{a,b}\cap\PP(c)} \pp^{r_{\pp}}\OO_{\pp}.
\]
In particular, 
\[
J_{a,b}=(I_{a,b}^a+I_{a,b}^a)\cap (I_{a,b}^b+I_{a,b}^b)
\]
and is diophantine over $K$. 
\end{lem}
\begin{proof}
 If $x,y\in I_{a,b}^c$, then
for any $\pp\in \Delta_{a,b}\cap\PP(c)$, 
\[
v_{\pp}(x)\equiv v_{\pp}(y) \equiv r_{\pp} \pmod{\ell}
\]
by Lemma~\ref{IabcCalculation}.
Also, $v_{\pp}(x)\geq r_{\pp}$ and $v_{\pp}(y) \geq r_{\pp}$, because $v_{\pp}(x),v_{\pp}(y)>0$.  Thus $x+y\in
\bigcap_{\pp\in \Delta_{a,b}\cap\PP(c)}\pp^{r_{\pp}}\OO_{\pp}$. 

Conversely, suppose that $z\in \bigcap_{\pp\in\Delta_{a,b}\cap\PP(c)}
\pp^{r_{\pp}}\OO_{\pp}$, so $v_{\pp}(z)\geq r_{\pp}$ for each $\pp\in \Delta_{a,b}\cap \PP(c)$. 
We will use weak approximation to show there exists $y\in K^{\times}$
such that $y,z-y \in I_{a,b}^c$. For the primes
$\pp\in \Delta_{a,b}$, we require the following:
\begin{enumerate}
\item if $\pp\in \Delta_{a,b}\cap\PP(c)$ such that
  $v_{\pp}(z)>r_{\pp}$, then $v_{\pp}(y)=r_{\pp}$;
\item if $\pp\in \Delta_{a,b}\cap \PP(c)$ and $v_{\pp}(z)=r_{\pp}$, let $p\in \pp\OO_{\pp}\setminus \pp^2\OO_{\pp}$ and write $z=p^{r_{\pp}}u$, with $u\in \OO_{\pp}^{\times}$.
\begin{enumerate}
\item If $u\not\equiv 1\pmod{\pp}$, 
   we require $y\equiv p^{r_{\pp}}(u-1) \pmod{\pp^{r_{\pp}+1}}$.
\item If $u\equiv 1 \pmod{\pp}$, then $z=p^{r_{\pp}}(1+k)$ for some $k\in\pp\OO_{\pp}$. Choose $m$ such that $m\ell>v_{\pp}(k)$. 
We take $y\equiv p^{r_{\pp}}(1+k+p^{m\ell}) \pmod{\pp^{r_{\pp}+m\ell+1}}$.  
\end{enumerate}
\item For $\pp\in\Delta_{a,b}\setminus \PP(c)$,  we require $v_{\pp}(y)<\min\{0,v_{\pp}(z)\}$ and $v_{\pp}(y)\equiv 0 \pmod
  \ell$ for $\pp\in \Delta_{a,b}\setminus \PP(c)$.
\end{enumerate}
We claim that $y,z-y\in I_{a,b}^c$. For
$\pp\in \Delta_{a,b}\cap\PP(c)$, by construction we have that
$v_{\pp}(y)$ is positive and congruent to $r_{\pp}$ modulo $\ell$. For
$\pp\in\Delta_{a,b}\setminus \PP(c)$, (iii) ensures that $v_{\pp}(y) =
v_{\pp}(1-y) \equiv 0 \pmod{\ell}$. Thus $y\in I_{a,b}^c$. We will now show
that $z-y\in I_{a,b}^c$. First we claim that for $\pp\in
\Delta_{a,b}\cap \PP(c)$,  $v_{\pp}(z-y)>0$ and $v_{\pp}(z-y)\equiv v_{\pp}(c)
\pmod{\ell}$. We have that $v_{\pp}(z-y)\geq
\min\{v_{\pp}(z),v_{\pp}(y)\}>0$.  Additionally, we have that
$v_{\pp}(z-y)\equiv r_{\pp} \pmod{\ell}$ by (1) and (2) above, and by definition, $v_{\pp}(c)\equiv r_{\pp} \pmod{\ell}$. For
$\pp\in \Delta_{a,b}\setminus\PP(c)$, because
$v_{\pp}(y)\equiv 0 \pmod \ell$ and $v_{\pp}(y)<v_{\pp}(z)$, it
follows that 
$v_{\pp}(z-y)=v_{\pp}(y)\equiv 0 \pmod \ell$. Finally, from $v_{\pp}(z-y)<0$,
it follows that
\[
v_{\pp}(1-(z-y))=v_{\pp}(z-y)=v_{\pp}(y)\equiv 0 \pmod \ell.
\] 
We conclude that $z=y+(z-y)\in I_{a,b}^c+I_{a,b}^c$ by Lemma~\ref{IabcCalculation}. 

\end{proof}

\subsection{Partitioning the primes of $K$}
One main step in the proof that the non-norms are diophantine is
to partition the primes of a finite number of sets. We partition the
primes using ray classes for a fixed abelian extension of $K$, which
we now describe. The following proposition outlines the properties of
this extension. 
 
\begin{prop}\label{fixab}
  Suppose that $\ell$ is an odd prime number and that $K$ is a global
  field with $\ch(K)\neq \ell$. There exist $a,b\in K^{\times}$ satisfying the following
  conditions:
\begin{enumerate}[(i)]
\item $(a)$ and $(b)$ are coprime;
\item $a,b$ generate distinct, nontrivial subgroups of $K^{\times}/K^{\times\ell}$;
\item If $K$ is a number field, $a,b \in 1+\ell^{3}\OO_K$;
\item Set $L:=K(\sqrt[\ell]{a},\sqrt[\ell]{b})$.
  Suppose that $\qq$ is a prime of $K$ which is unramified in $L$.
 Set $A:=\OO_K$ if $K$ is a number field, and
  $A:=\OO_{\{\qq\}}$ if $K$ is a global function field. Given
  $\sigma\in \Gal(L/K)$, and an ideal class $ \mathcal{C} \in \Cl(A)$, there
  exists a prime $\pp$ of $K$ unramified in $L$
  such that $(\pp,L/K)=\sigma$ and $\pp\cap A\in \mathcal{C}$. 
\end{enumerate}
\end{prop}
\begin{proof}
First assume that $K$ is a number field. Fix a prime $\pp_0$ of $K$ which does not divide $\ell\OO_K$. There exists $a\in K^{\times}$ such that $a\in 1+\ell^{3}\OO_K$
and $v_{\pp_0}(a)=1$. To choose $b$, fix a different prime $\pp_1\neq \pp_0$ not
dividing $\ell\OO_K$. There exists $b\in
K^{\times}$  such that
\begin{itemize}
\item  $b\in 1+\ell^{3}\OO_K$,
\item  $v_{\pp}(b)=0$ for any $\pp|(a)$,
\item and $v_{\pp_1}(b)=1$. 
\end{itemize}
If $K$ is a global function field, similarly choose two primes
$\pp_0,\pp_1\not=\qq$ and $a,b\in K^{\times}$ such that
$v_{\pp_0}(a)=1$, $v_{\pp_1}(a)=0$, $v_{\pp'}(b)=0$ if $\pp'|(a)$, and
$v_{\pp_1}(b)=1$. 

Thus (i) is satisfied, because $b$ is chosen so that
$(b)$ is coprime to $(a)$.  We also have that (ii) holds because
$v_{\pp_0}(a)\not\equiv 0\pmod \ell$ and
$v_{\pp_1}(b)\not\equiv 0 \pmod \ell$, so both $a$ and $b$ generate
nontrivial subgroups of $K^{\times}/K^{\times \ell}$. These subgroups
are distinct, because and $v_{\pp_0}(a^m/b^n)\not\equiv 0 \pmod{\ell}$
for $1\leq m,n\leq \ell-1$. By construction, (iii) holds as well if
$K$ is a number field.

We now show that $(iv)$ holds if $K$ is a number field. Let $H$ denote the Hilbert class field of $K$. We claim that $H$ and $L$ are linearly disjoint. This holds
because all intermediate fields between $K$ and $L$ are ramified at
some prime of $K$. Indeed, they are of the form
$K(\sqrt[n]{a}),K(\sqrt[\ell]{b}),$ or $K(\sqrt[\ell]{ab^j})$ for some
$j=1,2,\ldots,\ell-1$ and each ramifies at either $\pp_0$ or $\pp_1$. Thus 
\[
\Gal(HL/K)\simeq \Gal(L/K)\times \Cl(K)
\]
so (iv) follows by the Chebotarev density theorem.

If $K$ is a global function field and $\qq$ is a prime of $K$ which is
unramified in $L$, the same argument works by replacing $H$ with
$K^A$, the maximal unramified abelian extension of $K$ in which $\qq$
splits completely. 
\end{proof}

We now fix $a,b\in K^{\times}$ as in Proposition \ref{fixab} and an
admissible modulus $\mm$ of $K$ for $L$ whose support contains all
primes in the support of $(\ell ab)$. 
We need to fix two extra constants
$c,d\in K$.
\begin{lem}\label{choosec,dforell}
Let $\omega\in K$ be a primitive $\ell$th root of unity. There exists $c,d\in K^{\times}$ such that
$(a,c)_{K_{\pp},\ell}=\omega$ for each prime $\pp\in \PP(a)$, 
$(b,d)_{K_{\qq},\ell}=\omega$ for each prime $\qq\in \PP(b)$, and 
\[
(\PP(a)\cup\PP(b))\cap \PP(c) = (\PP(a)\cup\PP(b))\cap \PP(d)= \emptyset.
\]

\end{lem}
\begin{proof} 
  Let $\pp_1,\ldots,\pp_r$ be the primes of $\PP(a)$. If
  $r\not\equiv 0\pmod{\ell}$, then choose primes
  $\pp_{r+1},\ldots,\pp_s$ such that for $r+1\leq t \leq s$, $\pp_t$ is coprime to $(a)$ and
  $(b)$, $\left(\frac{a}{\pp_t}\right)_{\ell}=\omega$, and $s \equiv 0 \pmod{\ell}$. Set
  $P:=\{\pp_1,\ldots,\pp_s\}$ in this case, and $P:=\PP(a)$ otherwise.
 For each prime $\pp \in \PP(a)\cup \PP(b)$, let $x_{\pp}\in
K^{\times}$ satisfy $v_{\qq}(x_{\pp})=0$ for each $\qq \in \PP(a)\cup \PP(b)$, and
$\left(\frac{x_{\pp}}{\pp}\right)_{\ell} = \omega$. Then we will use
  Theorem~\ref{prescribesymbols} to show there exists $c\in K^{\times}$
  such that 
\begin{itemize}
\item $(a,c)_{K_{\pp},\ell}=\omega$ for each $\pp\in P$;
\item $(a,c)_{K_{\pp},\ell}=1$ for all other primes $\pp$ of $K$; and
\item for each $\pp\in \PP(a)$, $(x_{\pp},c)_{K_{\qq},\ell} =1$ for
  all primes $\qq$ of $K$.
\end{itemize}
 It is clear
  that conditions (1) and (2) of Theorem~\ref{prescribesymbols} are
  satisfied, and now we show the existence of local elements. Let
  $\pp\in \PP(a)$ and let $k_{\pp}$ be the multiplicative inverse of
  $v_{\pp}(a)$ modulo $\ell$. Then we can choose $c_{\pp}\in K^{\times}$
  such that $v_{\qq}(c_{\pp})=0$ for each $\qq\in \PP(a)$ and such that
  $\left(\frac{c_{\pp}}{\pp}\right)_{\ell}=\omega^{k_{\pp}}$. Then
\[
(a,c_{\pp})_{K_{\pp},\ell}=(\omega^{k_{\pp}})^{v_{\pp}(a)}=\omega
\]
by Equation~\ref{hilbertformula}. For each $\qq\in \PP(a)$, because $x_{\qq}$ and $c_{\pp}$ are $\qq$-adic units, we have $(x_{\qq},c_{\pp})_{K_{\qq},\ell}=1$.  If
$\pp\in P\setminus \PP(a)$, we can choose any
$c_{\pp}\in K^{\times}$ such that $v_{\pp}(c_{\pp})=1$ and $v_{\qq}(c_{\pp})=0$ for any $\qq \in \PP(a)$. Then
\[
(a,c_{\pp})_{K_{\pp},\ell}=\left(\frac{a}{\pp}\right)_{\ell}=\omega,
\]
 again by Equation~\ref{hilbertformula}, 
 and
 \[
 (x_{\qq},c_{\pp})_{K_{\pp},\ell}=1
 \]
 because $x_{\qq},c_{\pp}$ are $\qq$-adic units for each $\qq\in \PP(a)$. For any prime $\pp\not \in
 P$, let $c_{\pp}$ be a $\pp$-adic $\ell$th power; then
 $(a,c_{\pp})_{K_{\pp},n}=1=(x_{\qq},c_{\pp})_{K_{\qq},\ell}$ for each $\qq\in \PP(a)$. Finally let $\pp$ be a prime not in $P$. Then we choose $c_{\pp}\in K$ which is a $\pp$-adic $\ell$th power, so $(a,c_{\pp})_{K_{\pp},\ell}=1=(x_{\qq},c_{\pp})_{K_{\qq},\ell}$ for each $\qq \in \PP(a)$. 
 
 Finally we observe the third item implies that for each
 $\pp\in \PP(a)\cup \PP(b)$, $\pp\not\in\PP(c)$ by
 Equation~\ref{hilbertformula}. The proof for the existence of $d$ is
 similar.
\end{proof}

We now fix a primitive $\ell$th root of unity $\omega \in K$ and constants $c,d\in K$ with the properties guaranteed by
Lemma~\ref{choosec,dforell}. Enlarge the modulus $\mm$ of $K$ for $L$
by any primes $\pp$ dividing $(c)$ or $(d)$ and any primes $\pp$ such
that $(a,c)_{K_{\pp},\ell}\not=1$ or $(b,d)_{K_{\pp},\ell}\not=1$.

We will connect the splitting behavior of
primes of $K$ in $L$ with ramification of cyclic algebras over
$K$. We identify $\Gal(L/K)\simeq \mu_{\ell}\times \mu_{\ell}$; let
$\omega$ be a generator of $\mu_{\ell}$. First, we partition
$\Gal(L/K)$ depending on whether the restriction of $\sigma\in
\Gal(L/K)$ to $K(\sqrt[\ell]{a})$ or $K(\sqrt[\ell]{b})$ is trivial or
not.

\begin{defn}$\left.\right.$
\begin{itemize}
\item $C_{(1,1)}:=\{(1,1)\}$;
\item $C_{(-1,-1)}:=\{(\omega^i,\omega^j): i,j\not=0\}$;
\item $C_{(1,-1)}:=\{(1,\omega^j):j\not=0\}$;
\item $C_{(-1,1)}:=\{(\omega^i,1):i\not=0\}$.
\end{itemize}
 \end{defn}
Now we partition the primes of $K$ depending on whether or not they
split completely in $K(\sqrt[\ell]{a})$ or $K(\sqrt[\ell]{b})$. For a
prime $\pp$ of $K$ and $p\in K^{\times}$ and for $i,j=\pm1$, set
\[
\PP^{i,j}:=\{\pp:\psi_{L/K}(\pp)\in C_{i,j}\}.
\]
and  
\[
\PP^{i,j}(p):=\PP(p)\cap \PP^{i,j}.
\]

\begin{prop}\label{DeltaComputation}
Suppose $a,b\in K^{\times}$ and a modulus $\mm$ are as in
\ref{fixab}. Given $p\in I(\mm)$, we have that 
\begin{align*}
\PP^{(-1,-1)}(p)&=\Delta_{a,p}\cap \Delta_{b,p}, \\ 
\PP^{(-1,1)}(p)&=\left(\bigcap_{k=0}^{\ell-1}\Delta_{ab^k,p}\right)
\cap \left(\bigcap_{k=1}^{\ell-1}\Delta_{a,c^kp}\right), \\ 
\PP^{(1,-1)}(p)&=\left(\bigcap_{k=1}^{\ell-1}\Delta_{a^kb,p}\right)
\cap \left(\bigcap_{k=1}^{\ell-1}\Delta_{b,d^kp}\right).
\end{align*}
\end{prop}

\begin{proof}
We will begin with the first equality. First, no primes $\pp|\mm$
occur in $\Delta_{a,p}\cap\Delta_{b,p}$, since if
$(a,p)_{K_{\pp},\ell}\not=1$ we must have that $\pp\nmid \ell$ and
$v_{\pp}(a)\not\equiv 0\pmod{\ell}$. But then, because $(a),(b)$ are
coprime, we have $(b,p)_{K_{\pp},\ell}=1$. Now suppose $\pp\nmid\mm$; then we 
have $v_{\pp}(a)=v_{\pp}(b)=0$. From Equation \ref{hilbertformula}, 
$(a,p)_{K_{\pp},\ell}=\left(\frac{a^{v_{\pp}(p)}}{\pp}\right)_{\ell}\neq 1$  if and
only if $\pp\in\PP(p)$ and $\left(\frac{a}{\pp}\right)_{\ell}\not=1$. 
Similarly, $(b,p)_{K_{\pp},\ell}=\left(\frac{b}{\pp}\right)_{\ell}^{v_{\pp}(p)}$ is not
$1$ if and only if $\pp\in\PP(p)$ and
$\left(\frac{b}{\pp}\right)_{\ell}\not=1$. Because 
$\psi_{L/K}(\pp)=\left(\left(\frac{a}{\pp}\right)_{\ell},\left(\frac{b}{\pp}\right)_{\ell}\right)$, 
for any $\pp\nmid\mm$, we have that
 $\pp \in \Delta_{a,p} \cap \Delta_{b,p}$ if and only if $\pp\in\PP^{(-1,-1)}(p)$.

 Now we will prove the second equality, as the proof of the third
 equality is similar to this case. Assume first that
 $\pp\in \PP^{(-1,1)}(p)$, we will show that $\pp$ is in the right
 hand side of the second equality. We compute
 $(a,p)_{K_{\pp},\ell}\not=1$ and $(b,p)_{K_{\pp},\ell}=1$. Thus for
 every integer $0\leq k \leq \ell-1$, we have
 $(ab^k,p)_{K_{\pp},\ell}\not=1$ and hence
 $\pp \in \bigcap_{k=0}^{\ell-1}\Delta_{ab^k,p}$. We also note that
 $(a,c)_{K_{\pp},\ell}=1$ by the construction of $c$ in the proof of
 Lemma~\ref{choosec,dforell}. Thus for $0\leq k \leq \ell-1$,
 $\pp\in \Delta_{a,c^kp}$.
 
 Now we will show the reverse inclusion. Suppose that $\pp|\mm$. We will show that $\pp$
 is not in the set on the right-hand side in the second
 equality. First, if $K$ is a number field and $\pp|\ell\OO_K$, we have that
 $(a,p)_{K_{\pp},\ell}=1$ because $a$ is a $\pp$-adic $\ell$-th
 power. If $\pp\in \PP(a)$, then $(a,c)_{K_{\pp},\ell}=\omega$. If additionally $(a,p)_{K_{\pp},\ell}\not=1$,
 then there is some $1\leq k\leq \ell-1$ such that 
\[
(a,c^kp)_{K_{\pp},\ell} = (a,p)_{K_{\pp},\ell} (a,c)_{K_{\pp},\ell}^k
= 1
\]
and thus $\pp\not\in \bigcap_{k=0}^{\ell-1} \Delta_{a,c^kp}$. If
$\pp|\mm$ but $\pp\not\in \PP(a)$, then $(a,p)_{K_{\pp},\ell}=1$ since
$v_{\pp}(a)\equiv 0 \pmod \ell$ and
$v_{\pp}(p)=0$. Thus $\pp\not\in \Delta_{a,p}$, so we conclude
that the right hand side in the second equality contains no primes
dividing $\mm$. If $\pp$ does not divide $\mm$ and $\pp\in
\bigcap_{k=0}^{\ell-1}\Delta_{ab^k,p}$ then $\pp\in
\PP^{(-1,1)}(p)$. Indeed, because $(a,p)_{K_{\pp},\ell}\not=1$, we
must have that $\pp\in \PP(p)$ and 
$\psi_{L/K}(\pp)\in C_{-1,-1}\cup C_{-1,1}$. We must have
$(b,p)_{K_{\pp},\ell}=1$, because otherwise for some $1\leq k \leq
\ell-1$ we would have $(ab^k,p)_{K_{\pp},\ell}=1$. Thus
$\psi_{L/K}(\pp)\in C_{-1,1}$, so $\pp\in \PP^{(-1,1)}(p)$. 

\end{proof}

\section{Controlling Integrality with Cyclic Algebras}\label{sec:integrality}

We maintain the notation of the previous section: $K$ is a global
field, $\ell$ is an odd prime, $K$ contains $\mu_{\ell}$,  we choose $a,b,c,d\in K^{\times}$
 and the extension $L/K$  so that they satisfy the properties guaranteed by Proposition~\ref{fixab} and
Lemma~\ref{choosec,dforell}, and $\mm$ is an admissible modulus of $K$ for $L$
containing the primes dividing $(\ell),(a),(b),(c)$ and $(d)$. Below, we define the semi-local rings of $K$ which are diophantine and whose primes will have certain splitting behavior in $L/K$. 

\begin{defn}\label{rpdefn}
Let $p,q\in K^{\times}$ and define 
\begin{align*}
R_p^{(-1,-1)}&:= T_{a,p}+T_{b,p},\\ 
R_p^{(-1,1)}&:= \sum_{k=0}^{\ell-1}T_{ab^k,p}+T_{a,c^kp},\\ 
R_p^{(1,-1)}&:= \sum_{k=0}^{\ell-1}T_{a^kb,p}+T_{b,d^kp}\\ 
R_{p,q}^{(1,1)}&:= T_{ap,q}+T_{bp,q}.
\end{align*}
\end{defn}

\begin{prop}\label{rpdio}
Let $p\in K^{\times}$ such that $(p) \in I(\mm)$. Then
\begin{align*}
R_p^{(-1,-1)}&=\bigcap_{\pp\in \PP^{(-1,-1)}(p)} \OO_{\pp},\\ 
R_p^{(-1,1)}&= \bigcap_{\pp\in \PP^{(-1,1)}(p)} \OO_{\pp}, \\ 
R_p^{(1,-1)}&= \bigcap_{\pp\in \PP ^{(1,-1)} (p)} \OO_{\pp}.\\ 
\end{align*}
\end{prop}
\begin{proof}
The proof follows from Definitions~\ref{rpdefn} and~\ref{def:Tab} and Corollary~\ref{DeltaComputation}.
\end{proof}

\subsection{Integrality at $\pp$ with $\psi_{L/K}(\pp)\not=(1,1)$}
We now define the sets which will let us parametrize a family of
diophantine semi-local rings whose primes all are contained in the
same union of ray classes for the modulus $\mm$. 
\begin{defn}
For $i,j=\pm 1$, define
\[
\Phi_{(i,j)}:=\{p\in K^{\times}: (p) \in I(\mm), \psi_{L/K}((p))\in C_{(i,j)}, \PP(p)\subseteq \PP^{(1,1)}\cup\PP^{(i,j)}\}.
\]
\end{defn}
\begin{defn}
Let $(i,j)\in \{(-1,-1),(1,-1),(-1,1)\}$. For $\pp\in \PP^{(i,j)}(p)$, define $r_{\pp}:=v_{\pp}(p) - \ell \lfloor \frac{v_{\pp}(p)}{\ell}\rfloor$ and
\[
J_p^{(i,j)} := \bigcap_{\pp\in \PP^{(i,j)}(p)} \pp^{r_{\pp}}\OO_{\pp}.
\]

\end{defn}

\begin{rmk}
  If $\pp\in \PP^{(i.j)}(p)$, then
  $v_{\pp}(p)\not\equiv 0 \pmod \ell$, so
  $r_{\pp}:=v_{\pp}(p) - \ell \lfloor \frac{v_{\pp}(p)}{\ell}\rfloor$
  satisfies $1\leq r_{\pp} \leq \ell-1$. Additionally, if
  $(i,j)\in \{(-1,-1),(1,-1),(-1,1)\}$, then $J_p^{(i,j)}$ is
  contained in the Jacobson radical $J(R_p^{(i,j)})$ of
  $R_p^{(i,j)}$. In particular, any element of $J_p^{(i,j)}$ is
  integral at all primes of $\PP^{(i,j)}(p)$.
\end{rmk}

\begin{lem}\label{cyclicnot1,1}$\left. \right.$
\begin{enumerate}[(a)]
\item For $i,j=\pm 1$, $\Phi_{(i,j)}$ is diophantine over $K$. 
\item For $(i,j)\not=(1,1)$, given $p\in \Phi_{(i,j)}$,
  $\PP^{(i,j)}(p)$ is nonempty. Furthermore, $J_p^{(i,j)}$ is diophantine. 
\item For $(i,j)\not=(1,1)$, Given $\pp$ with
  $\psi_{L/K}(\pp)\in C_{(i,j)}$, there exists $p\in \Phi_{(i,j)}$ with
  $\PP ^{(i,j)} (p)=\{\pp\}$. Additionally,
  $v_{\pp}(p) \equiv 1 \pmod{\ell}$, so
  $J_p^{(i,j)} = J(R_p^{(i,j)})$, the Jacobson radical of
  $R_p^{(i,j)}$.
\end{enumerate}
\end{lem}

\begin{proof}
To prove (a), we will first show that $\{p: (p)\in I(\mm)\}$ is
diophantine over $K$. First, the local rings $\OO_{\pp}$ are all 
diophantine over $K$. Choose $t\in K$ with $v_{\pp}(t)=1$; then 
$\pp^{\mm(\pp)}\OO_{\pp}=t^{\mm(\pp)}\OO_{\pp}$ is also
diophantine over $K$. We have that $K_{\mm,1}$ is diophantine 
over $K$ because 
\[
K_{\mm,1}=\bigcap_{\pp|\mm_0} 1+\pp^{\mm(\pp)}\OO_{\pp}
\]
and $\pp^{\mm(\pp)}\OO_{\pp}$ is diophantine over $K$ by Lemma 3.22 of~\cite{Shl94}. 
Next we claim that $\{p:(p)\in I(\mm)\}$ is  a
 finite union of $K^{\times}$-translates of $K_{\mm,1}$. This is because $C(\mm)$ is finite if $K$ is a number field, and because the subgroup of degree $0$ classes of $C(\mm)$ is finite if $K$ is a global function field.  Thus $\{p:(p) \in I(\mm)\}$ is diophantine over $K$ because $K_{\mm,1}$ is diophantine over $K$. 

 Now observe that $\PP(p)\subseteq \PP^{(1,1)}\cup\PP^{(i,j)}$ if and
 only if $\PP^{(i',j')}(p)=\emptyset$ for $(i',j')\not=(1,1),(i,j)$. This is equivalent to requiring that $p$ is in
 $(R_p^{(i',j')})^{\times}\cap (R_p^{(i'',j'')})^{\times}$, where $(i',j'),(i'',j'')$
 are the elements of $\{(\pm 1, \pm 1)\}$ which are not $(1,1)$ or
 $(i,j)$. Thus $\{p: \PP(p)\subseteq \PP^{(1,1)}\cup \PP^{(i,j)}\}$ is
 diophantine over $K$ because $(R_p^{(i',j')})^{\times}$ is diophantine
 over $K$ by Definiton~\ref{rpdefn} and Proposition~\ref{prop:Tabdio}.

Now we prove (b). Let $(p)=\prod \pp_s^{e_s}$ be the factorization of
$(p)$. Then for each $s$, we have $\psi_{L/K}(\pp_s)\in C_{(i,j)}$ or
$\psi_{L/K}(\pp_s)\in C_{(1,1)}$. Observe that for some $s$, we must
have that $\psi_{L/K}(\pp_s)\in C_{(i,j)}$ and  
$e_i\not\equiv 0 \pmod \ell$, because otherwise
$\psi_{L/K}((p))=(1,1)$. We conclude that $\pp_s\in \PP^{(i,j)}(p)$. We now use Lemma~\ref{lem:Jablemma} to show $J_p^{(i,j)} \subseteq J(R_p^{(i,j)})$ is diophantine over $K$. If $(i,j)=(-1,-1)$ then $\PP^{(-1,-1)}(p)=\Delta_{a,p}\cap\Delta_{b,p}$ and 
$J_p^{(-1,-1)}=J_{a,p}+J_{b,p}$. For $(i,j)=(-1,1)$,  we have that 
\[
\PP^{(-1,1)}(p) = \bigcap_{k=0}^{\ell-1}
\Delta_{ab^k,p}\cap\Delta_{a,c^kp}
\]
and hence
\[
J_p^{(-1,1)} = \sum_{k=0}^{\ell-1} J_{ab^k,p}+J_{a,c^kp}.  
\]
by Proposition~\ref{DeltaComputation} and
Definition~\ref{Jabdefforell}. We have a similar expression for
$J(R_p^{(1,-1)})$. Thus $J_p^{(i,j)}$ is diophantine by
Lemma~\ref{lem:Jablemma} for $(i,j)\not=(1,1)$. 

Finally we prove (c). If $K$ is a global function field, let $\pp_0\nmid \mm$ satisfy
$\psi_{L/K}(\pp_0)=(1,1)$, and set $A:=\OO_{\{\pp_0\}}$. If $K$ is a number field, let $A=\OO_K$. Suppose 
$\psi_{L/K}(\pp)\in C_{(i,j)}$ and let $\qq$
be an ideal of $K$ in such that $\qq\cap A$ is in the ideal class of
$(\pp\cap A)^{-1}$ in $\Cl(A)$ and such that 
$\psi_{L/K}(\qq)=(1,1)$; such an ideal exists by Proposition \ref{fixab}.
 Then 
\[
(\pp\cap A)(\qq\cap A)=pA
\]
for some $p\in K^{\times}$.  It follows that $(p)\in I(\mm)$,
$\psi_{L/K}((p))=\psi_{L/K}(\pp)\in C_{(i,j)}$, and
$\PP(p)\subseteq \PP^{(1,1)}\cup \PP^{(i,j)}$, so we conclude
$\PP^{(i,j)}(p)=\{\pp\}$. Additionally, $v_{\pp}(p)=1$, so $J_p^{(i,j)} = \pp\OO_{\pp}$. 
\end{proof}

\subsection{Integrality at $\pp$ with $\psi_{L/K}(\pp)=(1,1)$}

\begin{lem}\label{findq}
Suppose that $\pp_0,\qq_0$ are distinct primes of $K$ coprime to
$\mm$. Set $A:=\OO_K$ if $K$ is a number field, and
$A:=\OO_{\{\qq_0\}}$ if $K$ is a global function field. For any  $\zeta\in
\mu_{\ell}\subseteq K$, and $\sigma\in \Gal(L/K)$, there are infinitely many 
$q\in K^{\times}$ such that
\begin{enumerate}[(i)]
\item $(q)\in I(\mm)$ and $\psi_{L/K}(\qq)=\sigma$;
\item $qA$ is a prime ideal of $A$, so if $K$ is a global function
  field, there is a prime $\qq$ of $K$ such that $qA = \qq\cap A$;
\item $\left(\frac{q}{\pp_0}\right)_{\ell}=\zeta$. 
\end{enumerate}
\end{lem}
\begin{proof}
The proof is similar to that of Lemma~\cite{Park} if
$K$ is a number field or Lemma~3.15 of~\cite{EM17} if $K$ is a global
function field. We outline the proof here.

We have an isomorphism 
\[
K_{\mm}/K_{\mm,1} \to \left(A/\prod_{\pp|\mm}(\pp\cap A)^{v_{\pp}(\mm)}\right)^{\times}
\]
defined by writing $x=y/z$ for some $y,z\in A\cap K_{\mm}$ and then
mapping 
\[
y/z \mapsto yz^{-1}\mod \prod_{\pp|\mm}(\pp\cap A)^{v_{\pp}(\mm)}.
\]
 We then have
\[
K_{\mm\pp_0}/K_{\mm\pp_0,1}\simeq A/(\pp_0\cap A)\times K_{\mm}/K_{\mm,1}
\]
by the Chinese remainder theorem. We also note that
$K_{\mm}/K_{\mm,1}$ is isomorphic to the group of principal 
classes in $C_{\mm}$ via the mapping $p\mapsto (p)$. By 
Proposition \ref{fixab}, there is a prime $\qq'$ of $K$
such that $\qq'\cap A$ is in the principal ideal class and
$\psi_{L/K}(\qq')=\sigma$. Then there exists $q'\in K^{\times}$ such that
$q'A=\qq'\cap A$ and $\psi_{L/K}((q'))=\sigma$. Let $s\in K^{\times}$
satisfy $\left(\frac{a}{\pp_0}\right)_{\ell}=\zeta$. By the above
isomorphism, there is some $x\in K_{\mm}$ which maps to
$(q',s)$. Then there are infinitely many primes $\qq$ of $K$ such
$\qq\cap A$ is in the ideal class of $xA$ and thus are principle, say
of the form $qA$. Any such $q$
 satisfies the three properties of the
lemma.

\end{proof}

We will need the well-known fact that
$K_{\pp}^{\times}/K_{\pp}^{\times\ell}$ is finite for any prime $\pp$ of $K$ and
any prime number $\ell$. We sketch the proof below.
\begin{lem}\label{FiniteNonPowers}
Let $\ell$ be a prime number, let $K$ be a global field satisfying
$\ch(K)\neq \ell$, and let $\pp$ be a prime of $K$. Then
$K_{\pp}^{\times}/K_{\pp}^{\times\ell}$ is finite.
\end{lem}
\begin{proof}
Let $\pi\in K_{\pp}$ be a uniformizer for
 $R_{\pp}$. Then $K_{\pp}^{\times\ell}=\pi^{\ell}\cdot R_{\pp}^{\times\ell}$,
 so it suffices to show that $R_{\pp}^{\times\ell}$ has finite
 index in $R_{\pp}^{\times}$. Let $e$ be the absolute 
ramification index, meaning $\pi^eR_{\pp}=\ell R_{\pp}$.
 By Hensel's Lemma, 
if $\alpha\in 1+\pp^{2e+1}R_{\pp}$, then $\alpha$ is an 
$\ell$th power: let $f(x)=x^{\ell}-\alpha$, then
\[
  |f(1)|_{\pp}\leq
  \frac{1}{{\ell}^{2e+1}}<\frac{1}{{\ell}^{2e}}=\left|\ell\right|^2_{\pp}
  =|f'(1)|^2_{\pp}.
\]
Thus $1+\pp^{2e+1} R_{\pp}\subseteq R_{\pp}^{\times\ell}.$ Also,
$1+\pp^{2e+1}R_{\pp}$ is an open neighborhood of $1$ in the
profinite group $R_{\pp}^{\times}$ and hence has finite index. Thus it
also has finite index in $R_{\pp}^{\times\ell}$.
\end{proof}

We will next show that for a fixed prime $\pp$ of $K$,
$(x,y)_{K_{\pp},\ell}\neq 1$ 
cuts out a diophantine subset of $K^{\times}\times K^{\times}$. 
\begin{lem}\label{lem:hilbertdio}
Assume that $\ell$ is an odd prime, that $K$ is a global field with
$\ch(K)\neq \ell$, and $\pp$ is a prime of $K$. Then 
\[
\{(x,y)\in K^{\times}\times K^{\times}: (x,y)_{K_{\pp},\ell}\not=1\}
\]
is diophantine over $K$.
\end{lem}
\begin{proof}
There exist $s_1,\ldots,s_m\in K^{\times}$ which are a complete set of
representatives for $K_{\pp}^{\times}/K_{\pp}^{\times\ell}$
by Lemma~\ref{FiniteNonPowers}. If $\pi\in K_{\pp}$
is a uniformizer, let $e$ be the absolute ramification index, meaning
$\pi^eR_{\pp} = \ell R_{\pp}$. Now
 define $S_j:=s_j\cdot K^{\times\ell} \cdot (1+\pp^{2e+1}\OO_{\pp})$. We have that 
\[
K^{\times}=\bigcup_{j} S_j. 
\]
Given $x,y\in K^{\times}$, there exist $z_i,z_j\in K_{\pp}^{\times\ell}$ such
that $x=z_is_i$ and $y=z_js_j$. Then we compute 
$(x,y)_{K_{\pp},\ell}=(s_i,s_j)_{K_{\pp},\ell}$ by the linearity and
non-degeneracy of the Hilbert symbol. Thus
\[
\{(x,y): (x,y)_\pp\not=1\} = \bigcup_{i_j,i_k: (s_{i_j},s_{i_k})_{K_{\pp},\ell}\not=1} S_i\times S_j,
\]
and this set is diophantine over $K$. 
\end{proof}

We need the following sets for our diophantine definitions of semi-local rings in $K$ whose primes split completely in $L$. 
\begin{defn} $\left.\right.$
\begin{align*}
\widetilde{\Phi_{(i,j)}}&:=K^{\times\ell} \cdot \Phi_{(i,j)}. \\ 
\Psi&:=\left\{(p,q)\in \widetilde{\Phi_{(1,1)}}\times
\widetilde{\Phi_{(-1,-1)}}| \prod_{\pp | \mm} (ap,q)_{K_{\pp},\ell}\not=1 
\text{ and } p \in a^{\ell-1}\cdot K^{\times\ell}(1+J_p^{(-1,-1)})\right\}.
\end{align*}
\end{defn}

\begin{lem}\label{cyclic1,1} $\left.\right.$
\begin{enumerate}[(a)]
\item $\widetilde{\Phi_{(i,j)}}$ and $\Psi$ are diophantine over $K$. 
\item If $(p,q)\in \Psi$, then $\Delta_{ap,q}\cap \Delta_{bp,q}$ is
  nonempty. Moreover, the Jacobson radical $J(R_{p,q}^{(1,1)})$ contains $J_{p,q}^{(1,1)}$, which is diophantine over $K$. 
\item Given $\pp_0$ with $\psi_{L/K}(\pp_0)=(1,1)$, there exists
  $(p,q)\in \Psi$ such that $\Delta_{ap,q}\cap
  \Delta_{bp,q}=\{\pp_0\}$ and $v_{\pp_0}(p)\equiv 1 \pmod
  \ell$. Moreover, $J(R_{p,q}^{(1,1)})=J_{ap,q}+J_{bp,q}$ and thus is diophantine over $K$. 
\end{enumerate}
\end{lem}

\begin{proof}
We have that $\widetilde{\Phi_{(i,j)}}$ is diophantine over $K$ by
Lemma \ref{cyclicnot1,1} part (a), and together with Lemma \ref{cyclicnot1,1} part
(b) and Lemma \ref{lem:hilbertdio}, we conclude $\Psi$ is diophantine
over $K$ as well.

We now prove (b). Let $(p,q)\in \Psi$. Then there is some prime
$\pp\nmid \mm$ such that $(ap,q)\not=1$ by Hilbert Reciprocity. Then
either $v_{\pp}(ap)$ or $v_{\pp}(q)$ is not divisible by
$\ell$. Because $\pp\nmid \mm$, we have $\pp \in \PP(p)\cup \PP(q)$
and consequently $\pp\in \PP^{(1,1)}\cup \PP^{(-1,-1)}$. We claim
that $\pp\in \PP^{(1,1)}$, so we assume toward a contradiction that
$\pp\in \PP^{(-1,-1)}$. Then $\pp\in \PP^{(-1,-1)}(q)$, and $p \in
a^{\ell-1}\cdot K^{\times\ell}\cdot(1+\pp\OO_{\pp})$. Thus $ap\in
K_{\pp}^{\times\ell}$ by Hensel's Lemma, and $(ap,q)_{K_{\pp},\ell}=1$,
a contradiction. Thus $\pp\in \PP^{(1,1)}$. 

Now, we observe that $(a,q)_{K_{\pp},\ell}=(b,q)_{K_{\pp},\ell}=1$ and
hence $(p,q)_{K_{\pp},\ell}\not=1$. Thus $(bp,q)_{K_{\pp},\ell}\not=1$
as well. We conclude that $\pp\in\Delta_{ap,q}\cap\Delta_{bp,q}$. 

In this
case, $J_{ap,q}+J_{bp,q}$ is diophantine, and is contained in (but not
necessarily equal to) the Jacobson
radical of $R_{p,q}^{(1,1)}$. We have
\[
J_{ap,q}+J_{bp,q} = \bigcap_{\pp\in \Delta_{ap,q}+\Delta_{bp,q}} \pp^{r_{\pp}}\OO_{\pp},
\]
where $r_{\pp} = \max\{v_{\pp}(p) - \ell\lfloor\frac{v_{\pp}(p)}{\ell}\rfloor, v_{\pp}(q) - \ell\lfloor\frac{v_{\pp}(q)}{\ell}\rfloor\}.$

We move on to part (c). Suppose $\pp_0$ is a prime of $K$ with
$\pp\not|\mm$ and $\psi_{L/K}(\pp_0)=(1,1)$. We  will now construct 
$(p,q)\in \Psi$ 
such that $\Delta_{ap,q}\cap\Delta_{bp,q}=\{\pp_0\}$. We begin by 
choosing our candidate for $q$. If $K$ is a number field, set
$A:=\OO_K$, and if $K$ is a global function field, let $\pp_1\neq
\pp_0$ be a prime of $K$
such that $\psi_{L/K}(\pp_1)=(1,1)$ and set $A:=\OO_{\{\pp_1\}}$. Let $\zeta\in \mu_{\ell}$ be a primitive $\ell$th root of unity.
Then by Lemma \ref{findq}, there exists infinitely many  
$q\in K^{\times}$ such that $\psi_{L/K}((q))=(\zeta,\zeta)$, 
$\left(\frac{q}{\pp_0}\right)_{\ell}=\zeta$, and $qA$ is a prime ideal of
$A$, so $qA=\qq\cap A$ for some finite prime $\qq$ of $K$. We
 then have that $\{\qq\}=\Delta_{a,q}\cap\Delta_{b,q}$ by Lemma
 \ref{DeltaComputation}. 
We note that this choice of $q$ implies $q\in \widetilde{\Phi_{(-1,-1)}}$. 

For each $\pp|\mm$, by Lemma \ref{FiniteNonPowers}, there is a finite
generating set $E_{\pp}\subseteq K$ for $K_{\pp}^{\times}/K_{\pp}^{\times\ell}$.
Using the Chinese Remainder Theorem, we can assume that for each
$e\in E_{\pp}$, $e\equiv 1 \pmod{\pp_0}$. By Hensel's Lemma,
$e\in K_{\pp_0}^{\times\ell}$ for each $e$.  Finally, fix $e_0\in K^{\times}$ such that
$\left(\frac{e_0}{\pp_0}\right)_{\ell}=1$ and
$\left(\frac{e_0}{\qq}\right)_{\ell}=\zeta$.

We now construct $p$. Using Theorem
\ref{prescribesymbols}, there exists
$p\in K^{\times}$ with the following prescribed Hilbert symbols:
\begin{center}
\begin{tabular}{|c|c|c|c|}
\hline
 & $\pp_0$ & $\qq$ & \text{all other primes} \\ 
 \hline
 $e\in E_{\pp}$, $\pp|\mm$ & $1$ & $1$ & $1$ \\
\hline
$e_0$ & $1$& $1$& $1$ \\
\hline
$q$ & $\zeta$ & $\zeta^{\ell-1}$ & 1 \\ 
\hline
$a$ & $1$& $1$& $1$ \\ 
\hline
$b$ & $1$& $1$& $1$ \\
\hline 
\end{tabular}
\end{center}

Clearly, the Hilbert symbol is $1$ for almost all $\pp$. Also, 
across any row, the product of the symbols is $1$. We now must
 show that for each $\pp$, there exists an element of $K^{\times}$ 
satisfying the prescription. First we will do the $\qq$ column: we
claim that $a^{\ell-1}$ 
satisfies all the prescriptions. For each 
$x\in \{a,b,e_0\}\cup E_{\pp}$, $x$ and $a$ are $\qq$-adic units 
and hence $(x,a)_{K_{\qq},\ell}=1$. Since $\psi_{L/K}((q))=(\zeta,\zeta)$
and $v_{\qq}(q)=1$, we have 
that $(a^{\ell-1},q)_{K_{\qq},\ell}=\zeta^{\ell-1}$. For the $\pp_0$
column, we take $x\in \pp_0\setminus
\pp_0^{2}$, so that in particular, $v_{\pp}(x)=1$. We then have that 
$(x,q)_{K_{\pp_0},\ell}=\left(\frac{q}{\pp_0}\right)_{\ell}=\zeta$. For $\pp|\mm$, 
as mentioned above, $E_{\pp}\subseteq K_{\pp_0}^{\times \ell}$, so 
$(x,e)_{K_{\pp_0},\ell}=1$. For $e_0$, we compute 
$(x,e_0)_{K_{\pp},\ell}=\left(\frac{e_0}{\pp_0}\right)_{\ell}=1$. Finally, 
$(x,a)_{K_{\pp_0},\ell}=(x,b)_{K_{\pp_0}\ell}=1$ because 
\[
  (1,1)=\psi_{L/K}(\pp_0)=\left(\left(\frac{a}{\pp_0}\right)_{\ell},\left(\frac{b}{\pp_0}\right)_{\ell}\right).
\]
By Theorem \ref{prescribesymbols}, there exists $p\in K^{\times}$ satisfying
the prescribed Hilbert symbols above.

We now claim that $p$ has the following properties:
\begin{enumerate}
\item For each $\pp|\mm$ and $e\in E_{\pp}$, $(e,p)_{K_{\pp},\ell}=1$. 
\item $(e_0,p)_{K_{\pp_0},\ell}=1$ and $(q,p)_{K_{\pp_0},\ell}=\zeta$.
\item For $\pp\nmid\mm$, $(a,p)_{K_{\pp},\ell}=(b,p)_{K_{\pp},\ell}=1$. 
\item $(q,p)_{K_{\pp_0},\ell}=\zeta^{\ell-1}$. 
\item $\prod_{\pp|\mm} (ap,q)_{K_{\pp},\ell} = \zeta$. 
\end{enumerate}

Conditions (1) through (4) follow immediately from the above table.
 For condition (5), we compute
\begin{align*}
\prod_{\pp|\mm} (ap,q)_{K_{\pp},\ell} 
&=\prod_{\pp|\mm} (a,q)_{K_{\pp},\ell}(p,q)_{K_{\pp,\ell}} \\
&=\prod_{\pp|\mm} (a,q)_{K_{\pp},\ell}\\
&= \left(\prod_{\pp\nmid\mm} (a,q)_{K_{\pp},\ell}\right)^{-1} \\
&= (a,q)_{K_{\qq},\ell}^{-1} \\
&= \zeta^{\ell-1},
\end{align*}
where the second equality follows from $(p,q)_{K_{\pp}}=1$ 
for $\pp|\mm$ by the construction of $p$, 
the third from Hilbert Reciprocity, the fourth is a computation using
the fact that $\PP(q)=\{\qq,\pp_1\}$ and Equation
\ref{hilbertformula}, and the fifth from $\psi_{L/K}((q))=(\zeta,\zeta)$.  

We claim that $(p,q)\in \Psi$. First, we will show that
$v_{\pp}(p)\equiv 0 \pmod \ell$ for each $\pp|\mm$. For all
$e\in E_{\pp}$, a generating set for $K_{\pp}^{\times}/K_{\pp}^{\times\ell}$, we
have $(e,p)_{K_{\pp},\ell}=1$ by (1). Then by the non-degeneracy of
the Hilbert symbol as a pairing on $K_{\pp}^{\times}/K_{\pp}^{\times\ell}$, we
have that $p\in K_{\pp}^{\times\ell}$ and hence
$v_{\pp}(p)\equiv 0 \pmod \ell$ for each $\pp|\mm$. By weak
approximation, there exists $r\in K^{\times}$ such that
$v_{\pp}(r^{\ell}p)=0$ for each $\pp|\mm$, so we may assume
$(p)\in I_{\mm}$. From $(a,p)_{K_{\pp},\ell}=(b,p)_{K_{\pp},\ell}=1$
for each $\pp\nmid\mm$ and by \ref{hilbertformula}, it follows that
any prime $\pp$ dividing $(p)$ to a power
$e_{\pp}\not\equiv 0 \pmod \ell$ must satisfy
$\psi_{L/K}(\pp)=(1,1)$. Thus $\psi_{L/K}((p))=(1,1)$ as well, so
$p\in \widetilde{\phi_{(1,1)}}$. By construction of $q$, we have
$q\in \phi_{(-1,-1)}$, and by (5),
$\prod_{\pp|\mm} (ap,q)_{K_{\pp},\ell} \not=1$. Now we claim that
$ap\in K_{\qq}^{\times\ell}$. Because $e_0$ and $a$ are $\qq$-adic units,
$(e_0,a)_{K_{\qq},\ell}=1$, and by (2), $(e_0,p)_{K_{\qq},\ell}=1$, so
$(ap,e_0)_{K_{\qq},\ell}=1$. Thus $v_{\qq}(ap)\equiv 0 \pmod \ell$,
again by Equation \ref{hilbertformula}. Similarly, we have that
$(ap,q)_{\qq}=1$. Because $q$ and $e_0$ generate
$K_{\qq}^{\times}/K_{\qq}^{\times\ell}$, and again using the non-degeneracy of
the Hilbert symbol, we conclude that $ap\in K_{\qq}^{\times\ell}$.  We have
that $R_q^{(-1,-1)}=\OO_{\qq}$ and $J(R_q^{(-1,-1)})=\qq\OO_{\qq}$,
so by Hensel's Lemma, we conclude
$K^{\times}\cap K_{\qq}^{\times\ell} = 1+J(R_q^{(-1,-1)})$ and that
$ap \in 1+J(R_q^{(-1,-1)})$. Hence
$p\in a^{\ell-1}K^{\times\ell}(1+J(R_q^{(-1,-1)}))$, as claimed. Thus
$(p,q)\in \Psi$.

Finally, we show that $\Delta_{ap,q}\cap \Delta_{bp,q}=\{\pp_0\}$. 
Because $(a,q)_{K_{\pp_0},\ell}=(b,q)_{K_{\pp_0},\ell}=1$, and
$(p,q)_{K_{\pp_0},\ell}=\zeta$, we have 
\[
(ap,q)_{K_{\pp_0},\ell}=(bp,q)_{K_{\pp_0},\ell}=\zeta\not=1,
\] 
so $\pp_0\in \Delta_{ap,q}\cap\Delta_{bp,q}$. As observed above, 
$(ap,q)_{K_{\qq},\ell}=1$, so $\qq\not\in
\Delta_{ap,q}\cap\Delta_{bp,q}$. 
If $\pp\neq \pp_0,\qq$ is any other prime not dividing $\mm$, we have
$(q,a)_{K_{\pp},\ell}=1$ and $(q,p)_{K_{\pp},\ell}=1$, so
$(q,ap)_{K_{\pp},\ell}=1$. If $K$ is a number field and  $\pp|(\ell)$, then
because $a\in K_{\pp}^{\times\ell}$, we have
$(a,q)_{K_{\pp},\ell}=1$. Also, $(p,q)_{K_{\pp},\ell}=1$ by construction of 
$p$, so $(ap,q)_{K_{\pp},\ell}=1$. If $\pp|\mm$ but does not divide
$(\ell)$, we again have $(p,q)_{K_{\pp},\ell}=1$. On the other hand, 
at most one of $(a,q)_{K_{\pp},\ell}$ and $(b,q)_{K_{\pp},\ell}$ can
not equal $1$, because $(a)$ and $(b)$ are coprime. 
Thus $\pp\not\in \Delta_{ap,q}\cap\Delta_{bp,q}$, so $\{\pp_0\} =
\Delta_{ap,q}\cap \Delta_{bp,q}$.

\end{proof}

\section{Proof of Main Theorem}\label{sec:mainthm}

\begin{lem}\label{notpowermodp}
Assume that $\ell$ is an odd prime, that $K$ is a global field with
$\ch(K)\neq \ell$, and $\zeta\in K$ is a primitive $\ell$th root of unity. Let $\pp$ be a prime of $K$. Also assume $s\in K^{\times}$ satisfies
 $v_{\pp}(s)=0$ and $\left(\frac{s}{\pp}\right)_{\ell}=\zeta\not=1$. Then the set
\[
s\cdot K^{\times\ell} \cdot (1+\pp\OO_{\pp})
\]
are the elements of $K$ such that $v_{\pp}(x)\equiv 0 \pmod \ell$
 and there exists $t\in K^{\times}$ such that $v_{\pp}(xt^{\ell})=0$ and $\left(\frac{xt^{\ell}}{\pp}\right)_{\ell}=\zeta$. 
\end{lem}
\begin{proof}
Suppose $x\in s\cdot K^{\times \ell} \cdot (1+\pp\OO_{\pp})$. Then there exists
$p\in \pp\OO_{\pp}$ and $t\in K^{\times}$ such that $x=st^{\ell}(1+p)$. Thus
 $v_\pp(x)=\ell v_{\pp}(t)$, so $v_{\pp}(x/t^{\ell})=0$. We compute
\[
  \left(\frac{xt^{-\ell}}{\pp}\right)_{\ell}=\left(\frac{s(1+p)}{\pp}\right)_{\ell}=\left(\frac{s}{\pp}\right)_{\ell}.
\]
Conversely, suppose there exists $t\in K^{\times}$ with $v_{\pp}(xt^{\ell})=0$ and 
$\left(\frac{xt^{\ell}}{\pp}\right)_{\ell}=\zeta$. Then
$\left(\frac{xt^{\ell}s^{-1}}{\pp}\right)_{\ell}=1$, so there exists some $u\in
\OO_{\pp}$ with $xt^{\ell}s^{-1}\equiv u^{\ell} \mod \pp$. We conclude 
$x\in s\cdot K^{\ell}\cdot (1+\pp\OO_{\pp})$. 
\end{proof}

We will now prove Theorem \ref{thm:nonnormprime}, which states that 
if $\ell$ is an odd prime, $K$ is a global field with $\ch(K)\neq
\ell$, and $K$ contains $\mu_{\ell}$, then
\[
\{(x,y)\in K^{\times}\times K^{\times}: x \text{ is not a norm of } K(y^{1/{\ell}})\}
\]
is diophantine over $K$.

\begin{proof}
We have that, given $x,y\in K^{\times}$, $x$ is not a norm in
$K(\sqrt[\ell]{y})/K$ if and only if it fails to be a relative local norm 
by the Hasse norm theorem. This happens if and only if there exists
a prime $\pp$ of $K$ such that $(a,b)_{K_{\pp},\ell}\neq 1$.
 Fix $a,b\in K^{\times}$ and a modulus $\mm$ of $K$ for
 $L:=K(\sqrt[\ell]{a},\sqrt[\ell]{b})$ as in Proposition \ref{fixab}. Define 
$s_{(-1,-1)}:=a=:s_{(-1,1)}$ and $s_{(1,-1)}:=b$. We will show that
there is a prime $\pp$ of $K$ such that $(x,y)_{K_{\pp},\ell}\not=1$ 
if and only if one of the following conditions are satisfied:

\begin{itemize}
\item $\exists\, \pp|\mm \text{ such that } (x,y)_{K_{\pp},\ell}\not=1$,
\item $\bigvee_{(i,j)\not=(1,1)} \exists p \in \Phi_{(i,j)} \text{ such that}$
\begin{align*}
\Biggl(\Biggl(x\in\bigcup_{r=1}^{\ell-1} p^r \cdot K^{\times \ell} 
\cdot (R_{p}^{(i,j)})^{\times}\Biggr)
&\land \Biggl(\bigvee_{\substack{0\leq c\leq \ell-1 \\
                                    1\leq d\leq\ell-1}} x^cy^d \in 
\bigcup_{k=1}^{\ell-1}s_{(i,j)}^k \cdot K^{\times \ell} \cdot (1+J^{(i,j)}_p)\Biggr)\Biggr) \\ 
\lor \Biggl(\Biggl(y\in \bigcup_{r=1}^{\ell-1} p^r  \cdot K^{\times \ell}
  \cdot (R_{p}^{(i,j)})^{\times}\Biggr) 
&\land \Biggl(\bigvee_{\substack{1\leq c\leq\ell-1 \\
  0\leq d\leq\ell-1}} x^cy^d \in 
\bigcup_{k=1}^{\ell-1}s_{(i,j)}^k \cdot K^{\times \ell} \cdot (1+J^{(i,j)}_p)\Biggr)\Biggr),
\end{align*}
\item $\exists(p,q)\in \Psi_K \text{ such that } q\in (R^{(1,1)}_{p,q})^{\times} \text{ and }$ 
\begin{align*}
\Biggl(\Biggl(x\in \bigcup_{r=1}^{\ell-1} p^r \cdot K^{\times \ell} 
\cdot (R_{p,q}^{(1,1)})^{\times}\Biggr) 
&\land \Biggl(\bigvee_{\substack{0\leq c\leq\ell-1 \\
  1\leq d\leq\ell-1}} x^cy^d 
\in \bigcup_{k=1}^{\ell-1} q^k \cdot K^{\times \ell} 
\cdot (1+J_{ap,q}+J_{bp,q}))\Biggr)\Biggr) \\ 
\lor\Biggl(\Biggl(y\in \bigcup_{r=1}^{\ell-1} p^r 
\cdot K^{\times \ell} \cdot (R_{p,q}^{(1,1)})^{\times}) 
&\land \Biggl(\bigvee_{\substack{1\leq c\leq\ell-1 \\ 0\leq d\leq\ell-1}}
  x^cy^d \in \bigcup_{k=1}^{\ell-1}q^k \cdot K^{\times \ell} 
\cdot (1+J_{ap,q}+J_{bp,q})\Biggr)\Biggr).
\end{align*}
\end{itemize}

This will imply the theorem, because
the sets above are all diophantine over $K$ by Lemmas~\ref{cyclicnot1,1} and~\ref{cyclic1,1}, Definiton~\ref{rpdefn}, and Proposition~\ref{prop:Tabdio}. 
 
We now prove the claim. Suppose $x$ is not a norm of $K(\sqrt[\ell]{y})$; then there exists a
prime $\pp$ such that $(x,y)_{K_{\pp},\ell}\not=1$. We will show that one of the above conditions on $x$ and $y$ is satisfied. If $\pp|\mm$, then the first condition holds and we are done, so assume otherwise.

First suppose that $\pp\nmid\mm$ satisfies
$\psi_{L/K}(\pp)\in C_{i,j}$ with $(i,j)\not=(1,1)$. By
Equation~\ref{hilbertformula}, we have
that $v_{\pp}(x)\not\equiv 0 \pmod \ell$ or $v_{\pp}(y)\not\equiv 0\pmod{\ell}$. Assume first that $v_{\pp}(x)\not\equiv 0 \pmod{\ell}$. Then by
Lemma~\ref{cyclicnot1,1} part (b), there exists $p\in \Phi_{(i,j)}$
such that $v_{\pp}(p)=1$ and $\PP^{(i,j)}(p)=\{\pp\}$. Then there
exists $1\leq r\leq \ell-1$ such that
$v_{\pp}(x)\equiv v_{\pp}(p^r) \pmod \ell$. Observe that since
$\psi_{L/K}(\pp)\not=(1,1)$, we must have that
$\left(\frac{s_{(i,j)}}{\pp}\right)_{\ell}\not=1$.
Thus 
\[
x\in p^r\cdot K^{\times \ell} \cdot\OO_{\pp}^{\times}=p^r\cdot K^{\times \ell}\cdot (R_{p}^{(i,j)})^{\times},
\]
where the equality follows from Definition~\ref{rpdefn}. Since 
$(x,y)_{K_{\pp},\ell}\not=1$, we must have that 
$\left(\frac{x^{v_{\pp}(y)}y^{-v_{\pp}(x)}}{\pp}\right)_{\ell}\not=1$, which implies 
\[
x^{v_{\pp}(y)}y^{-v_{\pp}(x)}\in s_{(i,j)}^k\cdot K^{\times \ell} \cdot(1+J(R_p^{(i,j)}))
\]
for some $k=1,\ldots,n-1$ by Lemma \ref{notpowermodp}. Writing 
$v_{\pp}(x)=\ell q_1+r_1$, $v_{\pp}(y)=\ell q_2+r_2$ with $1\leq
r_1\leq \ell-1$, $0\leq r_2\leq \ell-1$, we have that
\[
x^{r_2}y^{\ell-r_1}\in s_{(i,j)}^k\cdot K^{\times \ell} \cdot(1+J(R_p^{(i,j)})).
\]
The argument for when $v_{\pp}(x)\equiv 0 \pmod \ell$ and
$v_{\pp}(y)\not\equiv 0 \mod \ell$ 
is similar. 

Now suppose $(x,y)_{K_{\pp},\ell}\not=1$ for a prime $\pp\in
C_{(1,1)}$. Then using Lemma~\ref{cyclic1,1} part (c), there exists
$(p,q)\in \Psi$ such that $\Delta_{ap,q}\cap \Delta_{bp,q}=\{\pp\}$. 
Moreover, by the proof of Lemma \ref{cyclic1,1} part (c), $q$ can be chosen
so that $v_{\pp}(q)=0$ and $\left(\frac{q}{\pp}\right)_{\ell}\not=1$, and $p$ 
satisfies $v_{\pp}(p)\not\equiv 0\pmod \ell$. Then 
\[
q\in \OO_{\pp}^{\times}=(R_{p,q}^{(1,1)})^{\times}, 
\]
and assuming first that $v_{\pp}(x)\not\equiv 0 \pmod \ell$, it
follows that there exists $1\leq r \leq \ell-1$ such that
\[
 x\in p^r\cdot K^{\times \ell} \cdot (R_{p,q}^{(1,1)})^{\times}.
\]
By arguing similarly as in the case with $\psi_{L/K}(\pp)\not=(1,1)$, we conclude 
\[
x^{c}y^{d}\in q^k\cdot K^{\times \ell} \cdot(1+J(R_{p,q}^{(1,1)})).
\]
The case when $v_{\pp}(y)\not\equiv 0 \pmod \ell$ again is similar.

Now, conversely, suppose one of the three conditions hold. 
There is nothing to show if the first condition holds, so suppose the 
second holds. Without loss of generality, for some
$(i,j)\not=(1,1)$, there exists $p\in \Phi_{(i,j)}$ such that
\begin{equation}\label{valxnonzero}
x\in p^r\cdot K^{\times \ell} \cdot (R_p^{(i,j)})^{\times}\subseteq p^r\cdot
K^{\times \ell} \cdot(\OO_{\pp})^{\times}.
\end{equation}
This implies  
\[
v_{\pp}(x)\equiv v_{\pp}(p^r)\not\equiv 0 \pmod \ell.
\]
By Lemma \ref{cyclicnot1,1} part (b),
$\PP^{(i,j)}(p)\not=\emptyset$ 
and thus contains some prime $\pp$; we will now compute
$(x,y)_{K_{\pp},\ell}$. 
 For some $0\leq c \leq \ell-1$ and $1\leq d,k\leq \ell-1$, we have that 
\begin{equation}\label{xynonpower}
\tag{$\star$} x^cy^d\in s_{(i,j)}^k\cdot K^{\times \ell} \cdot
(1+J(R_p^{(i,j)}))\subseteq s_{(i,j)}^k\cdot K^{\times\ell} \cdot (1+\pp\OO_{\pp}).
\end{equation}
It follows that
$(x,y)_{K_{\pp},\ell}\neq 1$ if and only if
$(x,y)^d_{K_{\pp},\ell}\neq 1$ since $1\leq d \leq \ell-1$; thus we
will show $(x,y)^d_{K_{\pp},\ell}\neq 1$.

We compute
\begin{align*}
(x,y)_{K_{\pp},\ell}^d&=(x,x)^c_{K_{\pp},\ell}(x,y)_{K_{\pp,\ell}}^d \\ 
 & =(x,x^cy^d)_{K_{\pp},\ell}. 
\end{align*}
The first equality follows because $(x,x)_{K_{\pp},\ell}=1$, and the second follows from  linearity of
  the Hilbert symbol. Possibly
 by multiplying $x^cy^d$ by an $\ell$th power of $K^{\times}$, we can
assume $v_{\pp}(x^cy^d)=0$. Because
$v_{\pp}(x^cy^d)=0$, we have that
$(x,x^cy^d)_{K_{\pp},\ell}=\left(\frac{x^cy^d}{\pp}\right)_{\ell}^{v_{\pp}(x)}$ by
Equation \ref{hilbertformula}. 

Because of Equation $\star$ and by Hensel's Lemma, 
there exists $z\in K_{\pp}^{\times\ell}$ such that
$x^cy^dz=s_{(i,j)}^k$. This implies
$(x,x^cy^d)_{K_{\pp},\ell}=(x,s_{(i,j)})_{K_{\pp},\ell}$, and
$(x,s_{(i,j)})_{K_{\pp},\ell}\neq 1$ by Lemma
\ref{notpowermodp}. Combing the above computations, we see that
$(x,y)_{K_{\pp},\ell}\neq 1$.

Now suppose the third condition holds. By Lemma~\ref{cyclic1,1} part (b),
 $\Delta_{ap,q}\cap\Delta_{bp,q}\not=\emptyset$ and so it contains
 some prime $\pp$. We claim that $(x,y)_{K_{\pp},\ell}=1$. Without
 loss of generality and by the same reasoning above,
$x\in p^r\cdot K^{\times \ell} \cdot(\OO_{\pp})^{\times}$, and 
$ x^cy^d\in q^k\cdot K^{\times \ell} \cdot(1+\pp\OO_{\pp})$
for some $c$ satisfying $0\leq c \leq \ell-1$, and $r,d,k$ satisfying $1\leq r,d,k\leq \ell-1$. Because 
$(ap,q)_{K_{\pp},\ell}\not=1$ and $v_{\pp}(q)=0$, we must have that 
$\left(\frac{q}{\pp}\right)_{\ell}\not=1$. The same argument and similar 
computations to the above give us that $(x,y)_{K_{\pp},\ell}\not=1$, 
now with $q$ playing the role of $s_{(i,j)}$. 
\end{proof}

We now prove Corollary~\ref{cor:nonnormsqfree} using Theorem~\ref{thm:nonnormprime}, the Hasse norm theorem and local class field theory.
\begin{proof}[Proof of Corollary \ref{cor:nonnormsqfree}]
Assume $n$ is a square-free natural number, $K$ is a global field with
$(\ch(K),n)=1$, and $K$ contains $\mu_n$, the $n$th roots of
unity. Let $n=\prod_{i=1}^r \ell_i$,
where the
 $\ell_i$ are the distinct $r$ primes dividing $n$. We will show
\[
\{(x,y)\in K^{\times}\times K^{\times}: x \text{ is not a norm of }
K(\sqrt[n]{y})/K\} 
\]
is equal to 
\[
 \bigcup_{i=1}^r \{(x,y)\in K^{\times}\times K^{\times}: x \text{ is not a norm of }
K(\sqrt[\ell_i]{y})/K\}.
\]
After showing this, Corollary \ref{cor:nonnormsqfree} follows from Theorem
\ref{thm:nonnormprime} and Theorem~1.3 of~\cite{EM17} if $2|n$, because
the finite union of diophantine sets is
diophantine. 

Given a cyclic extension $M$ of $K$ of degree $n$ and $\alpha\in K$,
we have that $M=K(\alpha^{1/n})$ is the compositum of the fields
$L_i:=K(\alpha^{1/\ell_i})$. We note that $M/K$, and hence each
$L_i/K$, satisfies the Hasse norm principle. Let $\mathcal{P}$ be a
prime of $M$ and let $\Pp_i$ be the corresponding prime below
$\mathcal{P}$ in $L_i$ for each $i$. Let $\pp$ be the prime of $K$
below $\mathcal{P}$. Then
\[
M_{\mathcal{P}}=(L_1)_{\Pp_i}\cdots (L_r)_{\Pp_r}
\]
Since $M_{\mathcal{P}}/K_{\pp}$ is cyclic and hence abelian, by local class field theory we have
\[
N_{M_{\mathcal{P}}/K_{\pp}}(M_{\mathcal{P}}^{\times}) = \bigcap_{i=1}^r N_{L_{\Pp_i}/K_{\pp}}(L_{\Pp_i}^{\times}).
\]
Each $L_i/K$ is cyclic, so they satisfy the Hasse norm principle; we thus conclude
\begin{align*}
N_{M/K}(M^{\times})&=K^{\times}\cap\bigcap_{\Pp} N_{M_{\mathcal{P}}/K_{\pp}}(M_{\mathcal{P}}^{\times}), \text{ by the Hasse norm principle;}\\
&=K^{\times}\cap\bigcap_{\Pp}\bigcap_{i=1}^r N_{L_{\Pp_i}/K_{\pp}}(L_{\Pp_i}^{\times}), \text{ by local class field theory;} \\ 
&= K^{\times}\cap\bigcap_{i=1}^r \bigcap_{\Pp_i} N_{L_{\Pp_i}/K_{\pp}}(L_{\Pp_i}^{\times}) \\ 
&= K^{\times}\cap\bigcap_{i=1}^r N_{L_i/K}(L_i^{\times}) \text{ again by the Hasse norm principle.}
\end{align*}
Thus
\[
\{(x,y)\in K^{\times}\times K^{\times} : x \text{ is not a norm of } K(y^{1/n})\} 
\]
is equal to
\[
 \bigcup_{i=1}^r \{(x,y)\in K^{\times}\times K^{\times} : x \text{ is not a norm of } K(y^{1/l_i})\}.
\]
This is a finite union of sets diophantine over $K$ by Theorem
\ref{thm:nonnormprime} if $l_i$ is odd and by Theorem~1.3 of
\cite{EM17} if
$l_i$ is $2$.
\end{proof}

\section{Non-$n$th powers are diophantine}

In this section, we prove Corollary \ref{cor:nonnthpowersdio}. This
was proved in \cite{CTG15} when $K$ is a number field. 
\begin{cor}
Let $n>1\in \NN$ and let $K$ be a global field with $(\ch(K),n)=1$. Then
$K^{\times}\setminus K^{\times n}$ is diophantine over $K$. 
\end{cor}
\begin{proof}
It suffices to prove this for $n$ prime and $K$ containing a primitive
$n$th root of unity, as observed in \cite{VAV2012,CTG15}, which we now assume. Because the Hilbert
symbol 
\[
(\cdot,\cdot)_{K_{\pp},n}: K_{\pp}^{\times}/K_{\pp}^{\times n}\times
K_{\pp}^{\times}/K_{\pp}^{\times n}\to \mu_n
\]
 is a non-degenerate
pairing, we have that $x\in K^{\times}\setminus K^{\times n}$ if
and only if there exists $y\in K^{\times}$ such that $(x,y)_{K_{\pp},n}\neq
1$. This holds if and only if there exists $y\in K^{\times }$ such that $x$ is not a norm
of $K(\sqrt[n]{y})/K$. Set 
\[
D:=\{(x,y)\in K^{\times}\times K^{\times}: x \text{ is not a norm of }
K(\sqrt[n]{y})/K\}
\]
Putting this together, we see
\[
K^{\times}\setminus K^{\times n} = \{x\in K^{\times}: \exists y\in K^{\times} \text{ s.t. } (x,y)\in D\},
\]
so $K^{\times}\setminus K^{\times n} $ is diophantine over 
$K$ by Theorem \ref{thm:nonnormprime}. 
\end{proof}

\section{Non-norms of cyclic extensions}\label{sec:norootsof1}
We will first prove Theorem \ref{thm:norootsof1} in the case that $n=\ell$ is a
prime, $K$
contains $\mu_{\ell}$, the primitive $\ell$th root of unity, and $\ch(K)\neq \ell$,
and then show how the theorem for $n>1$ square-free and fields $K$ not
containing $\mu_n$ follows. Recall that, for a finite field extension
$L/K$ of degree $n$, a {\em norm form for $L/K$} is a homogeneous polynomial $f$ of
degree $n$ in the $n$ variables $t_1,\ldots,t_n$ (thought of as ranging over $K$)
such that there is a $K$-basis $b_1,\ldots,b_n$ of $L$ satisfying
\[
f(t_1,\ldots,t_n) = \prod_{\sigma\in \Gal(L/K)} \sigma\left(\sum_{i}
  t_ib_i\right)
\]
for any $t_1,\ldots,t_n\in K$.

\begin{proof}[Proof of Theorem \ref{thm:norootsof1}]
Assume $\ell$ is a prime number, that $K$ is a global field with
$\ch(K)\neq \ell$, and that $K$ contains $\mu_{\ell}$, the primitive $\ell$th
roots of unity. Set $d=\binom{2\ell-1}{\ell}$ and 
\[
D:=\{(x,y)\in K^{\times}\times K^{\times} : x \text{ is not a norm of }
K(\sqrt[\ell]{y})/K\}.
\]
Consider a cyclic
extension $L=K(y^{1/\ell})$ of $K$ for $y\in K^{\times}$,
and the $K$-basis  $\{y^{(i-1)/\ell}\}$,
$i=1,\ldots,\ell$, of $L$. Then every other basis of $L/K$ is of the form 
\[
w_i = \sum_{j} b_{ij}y^{(j-1)/\ell}
\]
for some matrix $(b_{ij})\in \GL_{\ell}(K)$. Then
\[
N_{L/K}(t_1,\ldots,t_{\ell}) = N_{L/K}\left(\sum_i t_i\left(\sum_j b_{ij}y^{(j-1)/\ell}\right)\right),
\]
as a polynomial in the variables $t_i$, has coefficients $f_s\in K[y,b_{11},\ldots,b_{\ell\ell}]$.
 Thus 
\begin{align*}
(x,a_1,\ldots,a_d)\in D(\ell,K) \iff  &\exists b_{ij},y\in K \text{ for
  } i,j=1,\ldots,\ell \text{ such that } \\
  &\bullet  f_s(y,b_{11},\ldots,b_{\ell\ell}) = a_s, \\
 &\bullet (b_{ij})\in \GL_{\ell}(K), \text{ and }  
\\ &\bullet (x,y)\in D. 
\end{align*}
Thus $D(\ell,K)$ is diophantine over $K$, because each condition in the
above list defines a diophantine set over $K$. The first two are
clearly diophantine, and the third is by Theorem \ref{thm:nonnormprime} if
$\ell$ is odd, and ny Theorem~1.3 of~\cite{EM17} if $\ell=2$. This proves the theorem
in the case that $K$ contains $\mu_{\ell}$, the $\ell$th roots of unity.

Now suppose $K$ does not contain the $\ell$th roots of unity. Let
$\omega\in \overline{K}$ be a primitive $\ell$th root of unity in an
algebraic closure $\overline{K}/K$, and set
$M:=K(\omega)$. If $L/K$ is a cyclic extension of degree $\ell$, a basis of
$L/K$ is also a basis of $ML/M$, since $M$
 and $L$ are both Galois over $K$ and $M\cap L = K$. 
Thus a norm form for $L/K$ is also a norm form of $ML/M$ by viewing
the variables as ranging over $M$ rather than $K$. 

We now show that $D(\ell,K)$ is diophantine over $K$. We have that
$(x,\vec{a})\in D(\ell,K)$ if and only if there is a cyclic extension
$L/K$ such that $f_{\vec{a}}$ is a norm form of $L/K$, and if $\sigma$
generates $\Gal(L/K)$, the cyclic algebra $(\sigma,x)$ is not
split. Let $\tau$ be a generator of $\Gal(ML/M)$ with
$\tau|_{L}=\sigma$. Since the map
\begin{align*}
Br(K)&\to Br(M) \\ 
[A] &\mapsto [A\otimes_K L]
\end{align*}
is multiplication by $[M:K]=\ell-1$ on the level of Brauer classes, and since 
\[
(\sigma,x)\otimes_K M \simeq (\tau,x)
\]
we have that $(\sigma,x)$ is split if and only if $(\tau,x)$ is split,
as $\ell-1$ is coprime to $\ell$. Thus $(x,\vec{a})\in D(\ell,K)$
 if and only if $(x,\vec{a})\in D(\ell,M)$, which is diophantine over 
$M$. Let $g\in M[s,t_1,\ldots,t_d,,u_1\ldots,u_m]$ be a polynomial which 
gives a diophantine definition of $D(\ell,M)$ over $M$, meaning
 \[
(x',a_1',\ldots,a_d')\in M\times M^d \iff \exists r_1,\ldots,r_m\in M \text{ s.t. }
f(x',a_1',\ldots,a_d',r_1,\ldots,r_m)=0. 
\]
Such a polynomial exists by our proof of the theorem above for global
fields containing the $\ell$th roots of unity. Write the coefficients
of $f$ as $K$-linear combinations in $\omega^i,i=0,\ldots,\ell-1$ and let
$f_i\in K[s,t_1,\ldots,t_{d},u_1,\ldots,u_m]$ be the polynomial
which is the coefficient of $\omega^i$ in $f$. Then $f_0,\ldots,f_{\ell-1}$
give a diophantine definition of $D(\ell,K)$ over $K$. This proves the
theorem for a global field $K$ with $\ch(K)\neq \ell$, where $n=\ell$
is a prime.

Finally, if $n$ is square-free, $K$ is a global field with
$(\ch(K),n)=1$, we can reduce Theorem \ref{thm:norootsof1}
 to the case that
$n$ is prime just as we did in the proof of Corollary
\ref{cor:nonnormsqfree}. Let $n=\ell_1\cdots \ell_r$ be the
prime factorization of $n$; then cyclic extensions $L/K$ of degree $n$
are the compositum of cyclic extensions $L_i/K$ of degree $\ell_i$. 
An element $x$ of $K$ is not a norm of $L/K$ if and only if $x$ is not
a norm of $L_i/K$ for some $i$, $1\leq i \leq r$. Hence 
\[
D(n,K) = \bigcup_{i=1}^r D(\ell_i,K),
\]
and we see that $D(n,K)$ is diophantine over $K$ because each
$D(\ell_i,K)$ is diophantine over $K$ by the above argument for
$n=\ell$ prime. This finishes the proof of the theorem.
\end{proof}

\section*{Acknowledgements}
I would like to thank Philip Dittman for pointing out a
  mistake in a previous version of this work, and for his comments and
  suggestions. I would also like to thank the anonymous referees for
  their careful reading, corrections, and suggestions. 

\end{document}